\documentclass[a4paper,11pt]{amsart}
\usepackage[top=2.2cm,left=2.8cm,right=2.8cm,bottom=2.8cm]{geometry}
\usepackage[foot]{amsaddr}

\usepackage{hyperref}
\hypersetup{
	colorlinks,
	citecolor=green,
	filecolor=black,
	linkcolor=blue,
	urlcolor=blue
}
\hypersetup{linktocpage}

\usepackage{cite}
\usepackage{graphicx}
\usepackage{amsbsy}
\usepackage{latexsym}
\usepackage{amsfonts}
\usepackage{amssymb}
\usepackage{upgreek}
\usepackage{cleveref}
\usepackage[usenames]{color}
\usepackage{amsmath,amsthm}
\usepackage{enumerate}
\usepackage{mathrsfs}
\usepackage{stmaryrd}
\usepackage{dsfont}
\usepackage{varwidth}
\usepackage{pdfpages}
\usepackage{comment}

\usepackage{tikz}
\usetikzlibrary{positioning}
\usetikzlibrary{arrows,shapes,math}
\usetikzlibrary{arrows.meta,calc,positioning,fit, matrix}
\usetikzlibrary{decorations.pathreplacing}
\usepackage{pgfplots}
\pgfplotsset{compat=1.18}
\usepgfplotslibrary{groupplots}

\DeclareMathOperator{\rank}{rank}

\providecommand{\abs}[1]{\lvert#1\rvert}

\providecommand{\Bigabs}[1]{\Bigl\lvert#1\Bigr\rvert}

\providecommand{\norm}[1]{\lVert#1\rVert}
\providecommand{\bignorm}[1]{\bigl\lVert#1\bigr\rVert}

\providecommand{\biggnorm}[1]{\biggl\lVert#1\biggr\rVert}

\renewcommand{\Re}{\operatorname{Re}}

\newtheorem{theorem}{Theorem}
\newtheorem{lemma}[theorem]{Lemma}
\newtheorem{prop}[theorem]{Proposition}
\newtheorem{cor}[theorem]{Corollary}

\theoremstyle{definition}

\theoremstyle{remark}
\newtheorem{remark}[theorem]{Remark}

\numberwithin{equation}{section}
\numberwithin{theorem}{section}

\theoremstyle{plain}

\newcommand{\cM}{{\mathcal{M}}}

\newcommand{\cF}{{\mathcal{F}}}
\newcommand{\cH}{{\mathcal{H}}}

\newcommand{\cI}{\mathcal{I}}
\newcommand{\cS}{\mathcal{S}}

\newcommand{\cW}{\mathcal{W}}

\newcommand{\Chi}{\raise .3ex
	\hbox{\large $\chi$}} 

\newcommand{\T}{{\mathbb{T}}}
\newcommand{\R}{\mathbb{R}}

\newcommand{\N}{\mathbb{N}}

\newcommand{\C}{\mathbb{C}}

\newcommand{\bbone}{\mathds{1}}

\usepackage[section]{algorithm}
\usepackage{algorithmicx}
\usepackage{algpseudocode}

\usepackage{todonotes}

\newcommand{\comp}{\mathbf{v}}
\newcommand{\exsol}{\mathbf{u}}
\newcommand{\fixedp}{\widehat{\mathbf{v}}}
\newcommand{\modfp}[1]{\mathbf{v}^{#1}}
\newcommand{\fpmap}{\cF}
\newcommand{\cfpmap}{\Phi}
\newcommand{\softt}{\mathcal{S}}
\newcommand{\recomp}{\mathcal{R}}
\newcommand{\nqp}{Q}
\newcommand{\mat}{\mu}

\newcommand{\compn}[1]{\comp^{(#1)}}
\newcommand{\exsoln}[1]{\exsol^{(#1)}}
\newcommand{\fixedpn}[1]{\fixedp^{(#1)}}

\newcommand{\initsol}{\widetilde{u}}

\newcommand{\normJ}[1]{\lVert#1\rVert_\nqp}

\newcommand{\normWg}[1]{\lVert #1\rVert_{\mathcal{W}(t_0,t_1)}}

\newcommand{\inp}[2]{\left<#1,#2\right>}

\newcommand{\altil}{\widetilde{\alpha}}
\newcommand{\albar}{\overline{\alpha}}
\newcommand{\timeint}{\mathcal{I}}

\newcommand{\Rbb}{\mathbb{R}}

\title[Iterative thresholding low-rank time integration]{Iterative thresholding low-rank time integration for high-dimensional problems}
\author{Markus Bachmayr, Tianyu Jin, Polina Sachsenmaier and Federico Vismara}
\address{Institut f\"ur Geometrie und Praktische Mathematik, RWTH Aachen University, Templergraben 55, 52062 Aachen, Germany}
\email{\{bachmayr,jin,sachsenmaier,vismara\}@igpm.rwth-aachen.de}

\thanks{The authors acknowledge funding by the European Union (ERC, COCOA, 101170147).}

\begin{document}
	
	\maketitle
	
	\vspace{-18pt}
	\begin{abstract}
		This work analyzes a method for time integration of high-dimensional linear Schr\"odinger-type problems based on hierarchical tensor approximations. In particular, this method provides a balance between error bounds and associated approximation ranks, using a scheme for iterative refinement with soft thresholding of tensors. The practical performance of the method is illustrated by numerical tests on coupled oscillators.
		
		\smallskip
		\noindent \emph{Keywords.}  low-rank approximation, hierarchical tensor format, time-dependent Schr\"odinger equations, quasi-optimal ranks
		\smallskip
		
		\noindent \emph{Mathematics Subject Classification.} Primary 65D40, 65F55, 65M12; Secondary 65Y20, 65L70
	\end{abstract}
	
	\section{Introduction}
	
	We build on the previous work \cite{BDS:25} on time integration of low-rank matrix decompositions and adapt the basic approach of soft-thresholded iterative refinement of low-rank approximations to the case of higher-order tensors.
	In this generalization, we use the hierarchical tensor format \cite{HackbuschKuehn:09,Hackbusch:12tensorspaces}, with tensor trains \cite{Oseledets:11} as a notable special case, and soft thresholding of such low-rank tensors as in \cite{BS17}.
	As in \cite{BDS:25}, in our analysis we focus on the case of time-dependent Schr\"odinger equations and show that the method has certain desirable properties for such problems. 
	However, the basic concepts can also be generalized to other problem classes. While the basic construction of the method follows \cite{BDS:25}, here we use a different strategy for steering the rank reduction parameters and for the analysis in order to arrive at a favorable dependence on the dimensionality of the problem.
	
	\subsection{Time integration methods for higher-order tensors}
	
	Solving Schr\"odinger-type evolution problems in high dimensions is a problem of central importance in physics and chemistry. 
	Low-rank approximations have become a standard tool for handling many classes of such problems. In the literature, there are various approaches to generating such approximations for time-dependent problems that differ particularly in how the approximation ranks are chosen.
	
	Dynamical low-rank approximation \cite{KL07,KochLubich2010,LubichRohwedderSchneiderVandereycken:13}, which in a physics context is known as multi-configuration time-dependent Hartree methods \cite{Meyer1990,WangThoss:03,ConteLubich:10}, is based on a restriction of the dynamics to manifolds of tensors of fixed ranks.
    Robustness with respect to small singular values can be achieved by time integration based on the splitting of tangent space projectors \cite{LubichOseledets:14,Lubich:15,LubichOseledetsVandereycken:15}.

	Basis update and Galerkin (BUG) integrators, introduced for the matrix case in \cite{CL22}, are a variant that can also be naturally combined with rank adaptation \cite{CKL22}; the generalization to tensor networks is given in \cite{CLS23}. Other methods with rank adaptation based on basis enrichment are considered in \cite{Yang:20,DektorRodgersVenturi:21,Dunnett:21,AC25}.
	
	A different type of rank adaptation is performed by step truncation methods \cite{DRV21,RV:23,LCK25}, which combine standard time stepping schemes without restriction to low-rank manifolds with rank truncation. When this truncation is done up to a specified error bound, the reference time stepping scheme can be approximated up to any specified error.
	Another alternative for achieving controlled errors are space-time variational formulations, which have been combined with low-rank tensors for parabolic \cite{Andreev:12,BF24} and for Schr\"odinger-type problems \cite{DEG24,DEG:25}.
	
	A limitation of most existing methods is that the balance between accuracy and tensor ranks of approximations is generally not ensured. On the one hand, whereas dynamical low-rank and related integrators ensure prescribed ranks, this may incur a modelling error; on the other hand, for methods based on time stepping, truncation up to sufficiently small tolerances will ensure convergence, but how the produced ranks relate to those required for approximating the solution for the given error is then not entirely clear. A first approach on choosing truncation tolerances for step truncation methods, in order to guarantee a relation to the respective optimal approximation ranks, is described in \cite{DSZ26} for general matrix-valued ordinary differential equations when considering explicit time-stepping schemes. However, as shown there, factors in the rank bounds that scale as the inverse time step can generally be unavoidable for such methods.
	
	Sharper comparisons to best approximation ranks are possible based on space-time formulations as in \cite{BF24,DEG:25}, based on a gradual refinement of approximations on the entire time interval.
	Since the existence of such formulations can be restrictive, in this work we consider a hybrid approach where approximations are refined on subintervals (the size of which need not tend to zero) based on fixed-point formulations, following the basic strategy developed in our earlier work \cite{BDS:25}. 
	
	\subsection{Balancing ranks and approximation errors}
	
	High-dimensional evolution problems suffer from the curse of dimensionality: a standard tensor-product discretization with $n$ degrees of freedom per variable leads to $n^d$ unknowns in dimension $d$. Hierarchical low-rank tensor formats can reduce this complexity, provided that the solution admits accurate approximations with moderate hierarchical ranks. In the present work, we aim for the construction of a rank-adaptive algorithm in the sense that the ranks of the computed approximations should be comparable to the best ranks required to achieve the same accuracy.
	
	In contrast to matrices, whose complexity is described by a single rank parameter, hierarchical tensors are characterized by a collection of ranks associated with the matricizations determined by the chosen dimension tree. We measure the complexity of such a tensor by the largest rank occurring among these matrix unfoldings. For a reference object $u$ and a prescribed accuracy $\eta>0$, the corresponding optimal approximation rank is the smallest possible value of this maximal rank among all approximations of $u$ with error at most $\eta$. Ideally, the ranks produced by an adaptive method should be bounded, up to fixed constants, by these optimal ranks at comparable accuracies. Such rank bounds, measured relative to best approximations of the exact solution, are available for linear operator equations; see, for example, \cite{bachmayr_adaptive_2015,BS17,Bachmayr23}.
	
	Our construction follows the general strategy introduced for low-rank matrix evolution problems in \cite{BDS:25}. The solution is advanced on successive time subintervals by solving fixed point iterations. During these iterations, the approximation is gradually refined, while ranks are controlled by appropriately adjusted truncations. For this purpose we use soft thresholding of hierarchical tensors as a main ingredient, since this keeps convergence properties of fixed point iterations intact. In the matrix case, soft thresholding acts on the singular values and has favourable stability properties that are compatible with contraction arguments. In the hierarchical tensor case, we apply soft thresholding successively to the matricizations associated with the nodes of the dimension tree.
	
	Unfortunately, it turns out that tensorial soft thresholding is less stable than its matrix counterpart, since the successive thresholding of different matricizations can accumulate perturbations along the dimension tree. A straightforward adaptation of the tolerance choices from \cite{BDS:25} would thus lead to large dimension-dependent factors in the error and rank estimates. We avoid the unfavorable factors by using a different adjustment of truncation thresholds and approximation tolerances.
	
	Concerning the choice of the reference object $u$, for a time-dependent problem, there is more than one natural reference for 
	rank comparisons. One may compare the computed ranks with the best approximation ranks of the exact solution at a given time. However, the fixed point problem solved on a subinterval is also determined by the numerical data inherited from previous subintervals, and hence by a perturbed initial value. Its solution may have different approximation ranks from the exact solution. We therefore analyze rank bounds with respect to both the exact solution and these perturbed local fixed points. This separates the intrinsic low-rank complexity of the evolution from the additional rank growth caused by error propagation.

	As a particularly relevant guiding example, we consider in this work linear, high-dimensional Schr\"odinger initial value problems with bounded potentials $V_t$ of the form
	\begin{equation}
		\label{eq:schroedinger}
		i\partial_t u = -\Delta u +  V_t u \qquad \text{for $t \in (0,T]$,}
	\end{equation}
	with initial datum $u(0) = u_0$ on the spatial domain $\Omega \subseteq\Rbb^d$.
	To bring this into a suitable fixed-point form, we use a formulation based on \textit{Duhamel's formula}
	\begin{equation}\label{eq:duhamel}
		u(t) = e^{it \Delta}u_0 -i \int_{0}^{t} e^{i(t-s)\Delta} V_s u(s) \, ds,
	\end{equation}
	where the right-hand side has the requisite Lipschitz continuity properties. Our basic approach can also be adapted to other problem classes, such as parabolic problems, based on suitable similar fixed-point formulations, see \cite[Sec.~5]{BDS:25}. Here, however, we also focus in particular on the favorable properties of our scheme, such as norm preservation, for the considered Schr\"odinger equations.
	
	\subsection{Outline and notation}  
	
	After presenting the setting of the problem and recalling some preliminary facts about hierarchical tensor approximation in Section \ref{sec:preliminaries}, we recapitulate the construction of the scheme and some basic accuracy bounds in Section \ref{sec:thresholdingscheme}.
	The convergence analysis is subdivided into two parts: in Section \ref{sec:ranksdiscr}, we analyze the approximation ranks produced by the method in relation to those of best approximations of the exact next steps in the time integration scheme. Although this is a natural reference quantity, it depends on the error history over several temporal subintervals. 
	In Section \ref{sec:comparison_wrt_exact}, we thus consider rank estimates in relation to the exact time evolution, which leads to some additional difficulties.
	We report on the results of first numerical tests on an interacting many-particle quantum system in Section \ref{sec:numexp}. We write $A \lesssim B$ to denote that there exists a constant $C>0$ such that $A\leq CB$, and $A\eqsim B$ for $A\lesssim B$ and $B\lesssim A$.

	\section{Preliminaries}\label{sec:preliminaries}
	
		\subsection{Schr\"odinger initial value problems}
	
	Consider a generic evolution problem of the form
	\begin{align}\label{eq:ivp}
		\begin{split}
			\partial_t u &= G_t u(t), \qquad t\in(0,T], \\
			u(0) &= u_0.
		\end{split}
	\end{align} 
	The functional setting is the complex Hilbert space $L^2(\Omega) = L^2(\Omega,\mathbb{C})$ on $\Omega = \Omega_1 \times\cdots\times \Omega_d \subseteq\Rbb^d$ with the inner product
	\begin{equation*}
		\inp{u}{v} = \int_\Omega \overline{u(x)} v(x)\,dx.
	\end{equation*}
	We also write $\norm{\cdot} = \norm{\cdot}_{L^2(\Omega)}$ and use the same notation for the corresponding operator norm.
	The operator $G_t:L^2(\Omega)\to L^2(\Omega)$ is linear and potentially time-dependent. We assume that its norm is bounded uniformly in time by $C_G>0$, that is,
	\begin{equation}\label{eq:unifbound}
	  \norm{G_t} \leq C_G \qquad \forall t\in[0,T].
	\end{equation}
    The integral formulation associated with \eqref{eq:ivp} then reads
	\begin{equation}\label{eq:int_form}
		u(t) = u_0 + \int_0^t G_s u(s)\, ds, \qquad t\in[0,T].
	\end{equation}
	
	The Duhamel formulation of the Schr\"odinger initial value problem \eqref{eq:duhamel} is not precisely of this form, but it is equivalent to the integral formulation \eqref{eq:int_form} for the \emph{twisted variable} $e^{-it\Delta}u(t)$ for the operator
	\begin{equation}\label{eq:twisted_f}
		G_t = -ie^{-it\Delta}V_te^{it\Delta}
	\end{equation}
	that we use in what follows.
	Note that this operator satisfies \eqref{eq:unifbound} with  \[  C_G = \sup_t \,\norm{V_t}_{L^2(\Omega)\to L^2(\Omega)}<\infty.  \] 
	The twisted variable formulation is thus a suitable  starting point for our analysis of the Schr\"odinger problem: we take $G_t$ in \eqref{eq:ivp} to be defined by \eqref{eq:twisted_f}, and we again denote the twisted variable by $u$. Observe that $\norm{e^{-it\Delta}u(t)}=\norm{u(t)}$ for all $t$, so that the change of variable does not affect norms.
	Moreover, as considered in more detail in Section \ref{sec:thresholdingscheme}, the action of $e^{it\Delta}$ for any $t \in \R$, on product domains with suitable boundary conditions, leaves the ranks in tensor decompositions unchanged. 
	In particular, one can easily recover the solution of the original problem \eqref{eq:schroedinger} from the twisted variable by inverting the transformation.
	
	Our approach is related to exponential integrators \cite{KT05}, in particular in a form originally suggested by \cite{Lawson67}. 
	However, standard exponential integrators (as adapted to fixed-rank manifolds in \cite{CV:23,SV:24}; see \cite{HO:10} for a general overview) generally rely on inverses of operators or Krylov space methods in the approximation of exponentials, whereas in our setting, it is more natural to decompose the arising exponentials directly as Kronecker products of lower-dimensional exponentials.
			
	\begin{remark}
	While we put some emphasis on norm-preserving properties of the methods considered here, this is not essential for their construction.
	Much of the basic analysis carried out here can also be generalized to right-hand sides induced by potentially nonlinear mappings that satisfy a Lipschitz condition uniformly in time. 
	\end{remark}
	
	Motivated by the global rank bounds derived in \cite{BDS:25}, which scale proportionally to $h^{-3}$, the aim is to construct a higher-order time stepping scheme and to then be able to use larger stepsizes. The idea is thus to use a superconvergent Gauss method, which can achieve up to double the order of convergence of the used intermediate stages. Since we are mainly considering Schrödinger equations, the natural choice seems to be the Gauss-Legendre method, since it preserves quadratic invariants such as the solution norm over time, see \cite{HairerLubichWanner:03} as well as \cite[Lemma 15]{BDS:25} for an adapted proof of isometry preservation.
	
	The resulting implicit Runge-Kutta method is solved using a Picard fixed-point iteration. This approach moreover allows us to maintain rank control on the iterates by using the soft-thresholding operator. Standard approaches using the integral formulation without going to twisted variables typically yield a step size restriction which depends strongly on the spatial discretization, see also \cite[Example 1]{BDS:25}. 
	
	\subsection{Fixed point formulation}\label{subs:fixed_point_form}
	Consider a generic time subinterval of the form $[t_0,t_1]$, with $t_1-t_0=h$.
	For $j = 1, 2, \dots, \nqp$, let $v_j\in L^2(\Omega)$ be approximations of the exact solution $u(\tau_j)$ at $\nqp$ time instants $t_0\leq\tau_1<\dots<\tau_\nqp\leq t_1$, where $\tau_j=t_0+c_jh$ for some $c_j\in[0,1]$. We define the vector $\mathbf{v}=(v_1,\dots,v_\nqp)\in L^2(\Omega)^\nqp$ and we equip the space $L^2(\Omega)^\nqp$ with the norm
	\begin{equation*}
		\normJ{\mathbf{v}} = \max_{j=1,\dots,\nqp}\norm{v_j}.
	\end{equation*}
	For later use, we also define the space $\cW(t_0,t_1):=L^\infty((t_0,t_1),L^2(\Omega))$, with the norm
	\begin{equation*}
		\normWg{u} = \sup_{t\in[t_0,t_1]}\norm{u(t)}.
	\end{equation*}
	
	Assume that an approximation $\initsol_0\in L^2(\Omega)$ of $u(t_0)$ is given at the initial time. We then define the map $\cfpmap_{\initsol_{0}}:L^2(\Omega)^\nqp\to\cW(t_0,t_1)$ as
	\begin{equation}\label{eq:continuous_fp_map}
		\cfpmap_{\initsol_0}(\mathbf{v})(t) := \initsol_0 + \int_{t_0}^t\sum_{q=1}^\nqp G_{\tau_q}(v_q)\ell_q(s)\,ds, \qquad t\in[t_0,t_1],
	\end{equation}
	where, for $q=1,2, \dots,\nqp$,
	\begin{equation*}
		\ell_q(s) = \prod_{p\neq q}\frac{s-\tau_p}{\tau_q-\tau_p} = \prod_{p\neq q} \frac{\frac{s-t_0}{h}-c_p}{c_q-c_p} =: \widehat{\ell}_q\left(\frac{s-t_0}{h}\right)
	\end{equation*}
	is the Lagrangian interpolation polynomial associated to $\tau_q$, and $\widehat{\ell}_q:[0,1]\to\Rbb$ is the Lagrangian polynomial associated to $c_q$. 
	We also define the map $\fpmap_{\initsol_0}:L^2(\Omega)^\nqp\to L^2(\Omega)^\nqp$ such that
	\begin{equation}\label{eq:fixed_point_map}
		\left(\fpmap_{\initsol_0}(\mathbf{v})\right)_j := \cfpmap_{\initsol_0}(\mathbf{v})(\tau_j), \qquad \text{for} \qquad j = 1, 2, \dots, \nqp.
	\end{equation}
	Observe that we can write
	\begin{equation*}
		\left(\fpmap_{\initsol_0}(\mathbf{v})\right)_j = \initsol_0+\sum_{q=1}^\nqp G_{\tau_q}(v_q)\omega_{j,q}, \qquad \text{where} \qquad \omega_{j,q}:=\int_{t_0}^{\tau_j}\ell_q(s)\,ds.
	\end{equation*}
	The weights $\omega_{j,q}$ do not depend on the endpoints of the particular time interval $[t_0,t_1]$, but only on the interval length, as it can easily be shown that
	\begin{equation*}
		\omega_{j,q} = h \int_0^{c_j}\widehat{\ell}_q(s)\,ds, \qquad j,q=1,2, \dots,\nqp.
	\end{equation*}
	It then holds that
	\begin{equation}\label{eq:normQ_vs_normW}
		\normJ{\fpmap_{\initsol_0}(\mathbf{v}) - \fpmap_{\initsol_0}(\mathbf{w})} \leq \normWg{\cfpmap_{\initsol_0}(\mathbf{v}) - \cfpmap_{\initsol_0}(\mathbf{w})} 
	\end{equation}
	for arbitrary $\mathbf{v},\mathbf{w}\in L^2(\Omega)^\nqp$. Moreover, 
	\begin{equation}\label{eq:lipsch}
		\normWg{\cfpmap_{\initsol_0}(\mathbf{v}) - \cfpmap_{\initsol_0}(\mathbf{w})} \leq C_G \Lambda_\nqp h\normJ{\mathbf{v} - \mathbf{w}},
	\end{equation}
	for all $\mathbf{v},\mathbf{w}\in L^2(\Omega)^\nqp$, where
	\begin{equation*}
		\Lambda_\nqp := \int_0^1\sum_{q=1}^\nqp\Bigabs{\widehat{\ell}_q(s)} \,ds,
	\end{equation*}
	see also \cite[Proposition 8]{BDS:25} for the derivation.
	Note that the constant $\Lambda_\nqp$ only depends on the number of quadrature nodes $\nqp$ and on the choice of the collocation points $c_j$.
	Setting
	\begin{equation}\label{eq:rhodef}
	\rho =C_G \Lambda_\nqp h,
	\end{equation}
	the map $\fpmap_{\initsol_0}$ is a contraction with respect to $\normJ{\cdot}$ for $\rho < 1$, and therefore it has a unique fixed point $\fixedp$ such that $\fixedp = \fpmap_{\initsol_0}(\fixedp)$. 
	
	Note that we also have the proportionality
	\begin{equation}\label{eq:resequiv}
  		(1 - \rho) \normJ{ \fixedp - \mathbf{w} }  \leq  \normJ{ \fpmap_{\initsol_0}(\mathbf{w}) - \mathbf{w} } \leq ( 1 + \rho) \normJ{ \fixedp - \mathbf{w} }
	\end{equation}	
	between the error and residual.
	
	Ideally, one would like to compute $\fixedp$ and then set the new approximation at time $t_1$ as
	\begin{equation}
		\cfpmap_{\initsol_0}(\fixedp)(t_1) = \initsol_{0} + \sum_{q=1}^Q G_{\tau_q}(\widehat{v}_q)\omega_q, \qquad \text{where}\qquad \omega_q=\int_{t_0}^{t_1}\ell_q(s)\,ds. \label{eq: next_int_eval}
	\end{equation}
	It is well known that this procedure is equivalent to a collocation method: if $\fixedp=\fpmap_{\initsol_0}(\fixedp)$, then $\cfpmap_{\initsol_0}(\fixedp)$ is a polynomial of degree $\nqp$ in time satisfying $\partial_t\cfpmap_{\initsol_0}(\fixedp)(\tau_j) = G_{\tau_j} \cfpmap_{\initsol_0}(\fixedp)(\tau_j)$ for all $j = 1, 2 , \dots, \nqp$. In turn, this is equivalent to an implicit Runge-Kutta method \cite[Theorem 7.7]{hairer_solving_1993}. In particular, if the points $c_j$ are chosen to be the Gauss-Legendre quadrature points in $[0,1]$, then this procedure is equivalent to a $\nqp$-stage Gauss-Legendre method. Alternatively, if $c_j$ are the Gauss-Lobatto quadrature points in $[0,1]$, then we have a $\nqp$-stage Lobatto IIIA method. Note that in case of the Gauss-Lobatto quadrature we have $v_1 = \initsol_0$ as well as $v_{\nqp} = \cfpmap_{\initsol_0}(\fixedp)(t_1)$.

	\subsection{Low-rank hierarchical tensor approximations}\label{sec:lowranktensors}
	
	We consider low-rank representations for elements of a Hilbert tensor product space $\cH = \cH_1 \otimes\cdots \otimes\cH_d$ where each $\cH_i$ is a separable Hilbert space. The product space is endowed with the cross norm $\norm{\cdot}$ such that
	$\norm{ v_1 \otimes \cdots \otimes v_d } = \prod_{i = 1}^d \norm{v_i}_{\cH_{i}}$
	for all $v_1 \in \cH_1, \ldots, v_d \in \cH_d$. In the case of \eqref{eq:schroedinger}, we will be mainly interested in the case of $\mathcal{H}_i = L^2(\R)$ for $i = 1,2, \ldots, d$, corresponding to $\cH = L^2(\R^d)$.
	
	The hierarchical tensor (HT) formats that we consider here are based on \emph{matricizations} of a given tensor $v \in \cH$.
	Such a matricization is defined for each non-empty $s \subset \{1,2,\ldots, d\}$, where we set $s^{\mathsf c} = \{1, 2, \ldots,d\}\setminus s$, by interpreting $v$ canonically as a linear operator from 
	$\bigotimes_{i \in s} \cH_i$ to $\bigotimes_{i \in s^{\mathsf c}} \cH_i$, for which we can consider low-rank approximations based on its singular value decomposition.
	
	Subspace-based tensor formats are associated with distinct unordered pairs $\{ s_\mu, s^{\mathsf c}_\mu \}$, $\mu = 1, 2, \dots , E$, with non-empty $s_\mu, s_\mu^{\mathsf c} \subset \{1, 2, \ldots,d\}$ for some $E \in \N$, with corresponding matricizations
	\[
	\cM_\mu (v) \colon \bigotimes_{i \in s_\mu} \cH_i \to \bigotimes_{i \in s_\mu^{\mathsf c}} \cH_i, \quad \mu = 1, 2, \ldots, E.
	\]
	Here, $\cM_\mu(v)$ is called the \emph{$\mu$-matricization} of the tensor $v$. Since the pairs $\{s,s^{\mathsf c}\}$ and $\{s^{\mathsf c},s\}$ yield equivalent matricizations up to transposition, only one representative of each unordered pair is retained. Correspondingly, we denote by $\cM_\mu^{-1}$ the mapping that converts this matricization back to a tensor. The sequence of singular values of $\cM_\mu(v)$ is denoted by $\sigma_\mu(v) = (\sigma_{\mu,k}(v))_{k \in \N}$, and the hierarchical $\mu$-rank is 
	\[
	\operatorname{rank}_\mu(v) := \operatorname{rank}(\cM_\mu(v)) = \#\{k\in\N: \sigma_{\mu,k}(v) \neq0\}\;\in\;\N_0\cup\{\infty\}.
	\]
	
	The hierarchy of subsets of $\{1, 2,\dots,d\}$ can be described by a dimension tree, the nodes of which are subsets of $\{1, 2, \dots,d\}$. In this work, we restrict our attention to \emph{binary dimension trees}; detailed definitions and descriptions can be found in \cite{BS17,Bachmayr23}. 	
	For such a binary dimension tree $\T$, the number of associated nontrivial matricizations is $E=2d-3$.

	This yields a representation in terms of bases $\{U_k^{(i)}\colon k = 1, 2, \ldots, r_i \}$, $i = 1, 2, \dots, d$, in the leaves $\{1\},\ldots,\{d\}$ of $\T$, called \emph{mode frames}, and coefficient tensors $b^{(s)} \in \C^{r_s \times r_{s_1}\times r_{s_2}}$ (also called \emph{transfer tensors} or \emph{cores}) that are given for inner nodes $s \in \T$ of the tree with $\#s>1$ and $s = s_1 \dot\cup s_2$. Starting from the leaves, these coefficient tensors recursively define bases $\{U_k^{(s)} \colon k = 1,2,\ldots,r_s\}$ by
	$b^{(s)}_{k_0, k_1, k_2} = \langle U_{k_0}^{(s)}, U_{k_1}^{(s_1)}\otimes U_{k_2}^{(s_2)}\rangle$.
	
The number of stored entries required to represent a tensor $v$ of size $n_1\times\cdots\times n_d$ with components of these sizes is reduced to
	\[
	\sum_{i=1}^d r_i \, n_i + \sum_{\substack{s = s_1\dot\cup s_2\\ s_1, s_2 \in\T}} r_s r_{s_1} r_{s_2} .
	\]
	Since tensors of fixed rank are generally not closed under addition, hierarchical ranks may increase rapidly in numerical algorithms. Therefore, rank truncation is an essential operation for reducing such representations from larger to smaller ranks. 
	An essential property of hierarchical tensors is that a rank truncation with similar properties as the SVD truncation of matrices can be achieved by truncation of \emph{hierarchical SVD} (HSVD) representations. We say that a tensor is in HSVD form if for all $s \in \T\setminus \{1,2, \dots,d\}$, the associated basis vectors $\{ U_k^{(s)} \}$ are singular vectors of the corresponding matricizations.
		
	Unlike the truncated SVD in the matrix case, HSVD truncation does not in general provide a best approximation in the set of tensors with prescribed hierarchical rank. However, we have the following error bound and quasi-optimality property: let $\mathcal T_{k}(v)$ denote the truncation of an HSVD representation of $v$ to target ranks $k=(k_\mu)_{\mu=1}^E$,
	then
	\[
	\begin{aligned}
	\norm{v-\mathcal T_{k}(v)}
	&\le
	\biggl(
	\sum_{\mu=1}^E
	\sum_{j>k_\mu}
	\sigma_{\mu,j}(v)^2
	\biggr)^{\frac12}  \\
	&\le \sqrt{E} \min \bigl\{ \norm{ v - w} \colon w \in \cH, \; \rank_\mu(w) \leq k_\mu, \ \mu = 1, 2, \ldots, E \bigr\}.
	\end{aligned}
	\]

	The first of these two inequalities also provides a HSVD-based rank truncation up to a prescribed error $\delta>0$. To this end, we choose a vector of target ranks $k_\delta(v)=(k_{\delta,\mu}(v))_{\mu=1}^E$ such that $\max_\mu k_{\delta,\mu}$ is minimal while 
	\[ 
 	\sum_{\mu=1}^E \sum_{j>k_\mu}\sigma_{\mu,j}(v)^2 \leq \delta^2 \,.
	\]
	The corresponding recompression operator defined by
	\[
	  \recomp_\delta(v)=\mathcal T_{k_\delta(v)}(v)
	\] 
	then satisfies 
	$\norm{v-\recomp_\delta(v)}
    	 \leq
    	  \delta$; for further details, we refer to \cite{Grasedyck2010,Hackbusch:12tensorspaces,Bachmayr23}.
	  
    Soft thresholding for hierarchical tensors \cite{BS17} is an alternative technique for rank reduction that offers advantages in preserving the convergence properties of fixed-point iterations. Let $\cS_{\mu, \alpha}:\cH\rightarrow\cH$ denote the soft thresholding operation applied to the matricization $\cM_\mu(\cdot)$ of the input, that is,
	\[
	\cS_{\mu, \alpha}(v) = \bigl(\cM_\mu^{-1}\circ S_\alpha \circ \cM_\mu\bigr)(v),
	\]
	where $S_\alpha (X) := U\operatorname{diag}\bigl(\max\{\sigma_i(X)-\alpha, 0\}\bigr)V^*$	for a matrix $X = U\operatorname{diag}\sigma(X) V^*$. Then the hierarchical tensor soft thresholding operator $\mathbf \cS_\alpha: \cH\rightarrow\cH$ is defined as the composition
	\begin{equation}
		\label{eq:soft_thresh}
		 \cS_\alpha (v) = \bigl(\cS_{E,\alpha}\circ \cdots\circ \cS_{1,\alpha} \bigr) (v).
	\end{equation}
	As each matricization-wise soft thresholding map is non-expansive, the full composition $\cS_\alpha$ is non-expansive as well: for all $v,w\in \cH$, as shown in \cite[Prop.~3.2]{BS17} we have
	\begin{equation}
	\label{eq:soft_thresh_nonexp}
	    \norm{ \cS_\alpha(v) - \cS_\alpha(w) } \leq \norm{ v - w }.
	\end{equation}
	
	\section{Iterative thresholding time integration scheme}\label{sec:thresholdingscheme}
	
	For implementing time integration in low-rank format, on each subinterval we approximate the corresponding $\fixedp$ by fixed point iteration. Since for such an iteration in low-rank format, the ranks of iterates increase with every step, the basic iteration needs to be combined with a suitable mechanism for controlling the hierarchical ranks of the intermediate quantities. To this end, we use soft thresholding of hierarchical tensors as given in \eqref{eq:soft_thresh}. 
	
	\subsection{Basic construction of the method}
	
	The basic scheme for time integration follows the one analyzed for $d=2$ in \cite{BDS:25}. The time interval $[0,T]$ under consideration is subdivided into subintervals of length $h = T/N$ with $N \in \N$ such that $h$ is sufficiently small to ensure that for all $v \in L^2(\Omega)$, the mapping $\mathcal F_v$ defined in \eqref{eq:fixed_point_map} is a contraction. On each subinterval, given an initial value at the left interval boundary, we use Picard iteration with soft thresholding of hierarchical tensors to obtain an approximation of the value at the right boundary. The iteration is carried out on vectors $\mathbf{v} \in L^2(\Omega)^\nqp$ of approximations at the $\nqp$ intermediate times in the subinterval, and we define the appropriate soft thresholding operation by $\softt_{\alpha}\mathbf{v}\in L^2(\Omega)^\nqp$ with $(\softt_\alpha\mathbf{v})_j := \softt_\alpha v_j$ for $j=1,2,\dots, \nqp$. Since the transition to the starting value at the next subinterval is done using \eqref{eq: next_int_eval}, we combine this with an additional rank truncation step in order to take into account the rank increase from the summation.
	
	\subsubsection{Local iteration for each subinterval} \label{sec: local iter}
	We again assume a generic subinterval $[t_0,t_1]$ with $t_1 - t_0 = h$ with an initial value $\initsol_0 \in L^2(\Omega)$. We write $\fpmap=\fpmap_{\initsol_0}$ for notational simplicity, and we omit parentheses when applying $\fpmap$ or $\softt_{\alpha}$. We begin with a general description of the method, which involves several parameters. Their relation will be determined by the analysis of the scheme.
	
	The iterative scheme on each subinterval uses an outer loop with index $i$ and an inner loop with index $j$, with $i,j\geq 0$ and associated computed approximations $\comp_{i,j}$ at the start of the respective iteration. With the starting value $\comp_{0,0} \in L^2(\Omega)$ and suitably chosen $\alpha_0 >0$, we define the iteration by 
	\begin{equation}
		\comp_{i,j+1} = \softt_{\alpha_i}\fpmap\comp_{i,j}, \label{eq: iter def}
	\end{equation}
	for each $i \geq 0$, where $\alpha_i = \theta^i \alpha_0$ for a fixed $\theta \in (0,1)$. For a given $\nu\in(0,1)$, we decrease the threshold $\alpha_i$ adaptively by increasing the index $i$ when the condition  
	\begin{equation}\label{eq:thresholdcondition}
		\normJ{\comp_{i,j}-\comp_{i,j-1}}\leq\nu\frac{1-\rho}{\rho(1+\rho)}\normJ{\fpmap\comp_{i,j}-\comp_{i,j}}
	\end{equation}
	is satisfied.
	We denote by $J_i$ the smallest value of $j \geq 0$ for which this condition holds and set $\comp_{i+1,0} := \comp_{i,J_i}$.
	To start the algorithm, we initialize $\comp_{0,0}=0$, and choose $\alpha_0$ such that $\comp_{0,1}=\softt_{\alpha_0}\fpmap\comp_{0,0}=0$. While this is not the only possible choice, it is convenient for obtaining rank bounds, as we will show later.
	
	As the soft thresholding operator $\softt_{\alpha}$ is non-expansive for any value of the threshold $\alpha$, the combined map $\softt_\alpha\fpmap$ is still a contraction. Therefore, for each $i$, there is a unique fixed point denoted by $\modfp{\alpha_i}$ such that $\modfp{\alpha_i} = \softt_{\alpha_i}\fpmap\modfp{\alpha_i}$. Note that the strategy by which the threshold $\alpha$ is reduced differs from the one analyzed in \cite{BDS:25}: here, we reduce the threshold only when the algorithm detects convergence to the modified fixed-point $\modfp{\alpha_i}$ according to \eqref{eq:thresholdcondition}, whereas in \cite{BDS:25} the threshold is decreased in each iteration.
	
	For a prescribed tolerance $\varepsilon>0$, the algorithm stops when 
	\begin{equation*}
		\normJ{\comp_{i,j}-\fpmap\comp_{i,j}} < \varepsilon.
	\end{equation*} 
	We denote by $I$ the smallest value of $i \geq 0$ for which this condition is satisfied, and by $J_I$ the corresponding value of $j \geq 1$. The approximation at the subinterval boundary is then obtained from
	\begin{equation*}
		\initsol_1 = \recomp_{\delta} \bigl( \cfpmap_{\initsol_0} \comp_{I,J_I}(t_1) \bigr)
	\end{equation*}
	for a suitable rank truncation tolerance $\delta >0$. Note that the non-expansiveness property of the soft thresholding operator together with $\rho < 1$ guarantee the existence of such a finite index $I$. The existence of finite indices $J_i$ can be obtained from the proof of \cite[Theorem 5.1]{BS17}.
	
	\subsubsection{Low-rank realization of the fixed point mapping}
	
	For $t \in \R$ and $w\in\cH$,
	\[
	e^{it \Delta} (w_1 \otimes \cdots\otimes w_d) = (e^{it \partial^2_1} w_1) \otimes \cdots \otimes  (e^{it \partial^2_d} w_d),
	\]
	and as a consequence, $\rank_\mu(e^{it \Delta}w) = \rank_\mu(w)$ for all $\mu = 1, 2, \dots, E$. 
	Thus, the action of $e^{it\Delta}$ does not increase the hierarchical ranks. In a discrete tensor product basis, this operation amounts to applying the corresponding one-dimensional matrix exponential to each mode frame of the hierarchical tensor representation of $w$.
	
	It remains to describe the action of the potential term. We assume that, at all times, the operator $V_t$ admits a hierarchical representation with ranks bounded by $\rank(V_t) \in \N_0$ in terms of bounded linear operators on $\cH_j$, $j=1,\ldots,d$, in the leaves of the representation, so that for $w \in \cH$ with finite hierarchical ranks, we have $\rank_\mu(V_t w) \leq \rank(V_t) \rank_\mu(w)$ for $\mu = 1,\ldots,E$.
	
	Combining this with the mode-wise action of $e^{it\Delta}$, the operators $G_{\tau_q}, q=1, 2, \dots, Q$ appearing in the fixed point mapping $\fpmap_{\initsol_0}$ can be evaluated entirely in the HT format without assembling the full tensor. Since the potential term may increase the hierarchical ranks by a factor $\rank(V)$, the resulting tensors are recompressed after tensor additions in the fixed point iteration to maintain an efficient representation.
	
	\subsubsection{Global time integration} \label{sec:globaliter}
	
	In the following, we extend the notations of Section \ref{sec: local iter} to the global context. The complete algorithm is given in Algorithm~\ref{alg:basic}.
	
	The approximation $\initsol_0$ for the initial data is chosen such that $\normJ{\initsol_0-u(0)}\leq\eta_0$ for a prescribed initial accuracy $\eta_0$. With $t_n = nh$ and the constant step size $h = T/N$, we start by defining $\timeint_n=(t_{n-1}, t_n)$ as the $n$th time subinterval, for $n=1,2,\dots,N$. The collocation points in $\timeint_n$ are then $\tau_{j,n}:=t_{n-1} + c_jh$, for $j=1,2,\dots,\nqp$. We assume that an approximation $\initsol_{n-1}$ of the exact solution $u(t_{n-1})$ is given at $t_{n-1}$. The fixed point map \eqref{eq:fixed_point_map} in $\timeint_n$ is $\fpmap_{\initsol_{n-1}}=:\fpmap_n$, and we denote its unique fixed point as $\fixedpn{n}$, so that $\fixedpn{n}=\fpmap_n\fixedpn{n}$. We also use the same shorthand notation $\cfpmap_n:=\cfpmap_{\initsol_{n-1}}$ for the map \eqref{eq:continuous_fp_map}. The evaluations of the exact solution $u$ at the collocation points $\tau_{j,n}$ are collected in the vector
	\begin{equation*}
		\exsoln{n}=\bigl(u(\tau_{1,n}),\dots,u(\tau_{\nqp,n}) \bigr)\in L^2(\Omega)^\nqp.
	\end{equation*}\
	
	Here and in what follows, to avoid technicalities, we assume that $\fpmap_n$ is evaluated exactly. However, following the lines of \cite[Sec.~5.1]{BS17}, the scheme and its analysis could be adapted to inexact evaluation of the mapping with appropriately chosen relative residual tolerances.
	
	The fixed point iteration on $\mathcal I_n$ is then defined by $\compn{n}_{i,j+1}=\softt_{\alpha_{i,n}}\fpmap_n\compn{n}_{i,j}$. In each subinterval we initialize $\compn{n}_{0,0}=0$ and we choose $\alpha_{0,n}$ such that $\compn{n}_{0,1} = \softt_{\alpha_{0,n}}\fpmap_n\compn{n}_{0,0}=0$. For a prescribed $\nu\in(0,1)$, we decrease the threshold according to \eqref{eq:thresholdcondition}. Note that in this condition, in principle we could also let $\nu$ vary and take $\nu=\nu_n$.
	
	We then set $\alpha_{i+1,n}=\theta\alpha_{i,n}$ for a prescribed $\theta\in(0,1)$, and define $\compn{n}_{i+1,0} := \compn{n}_{i,J_i}$. The algorithm terminates on the respective subinterval once
	\begin{equation}\label{eq:stoppingn}
		\normJ{\compn{n}_{i,j}-\fpmap_n\compn{n}_{i,j}} < \varepsilon_n
	\end{equation}
	holds for a prescribed threshold $\varepsilon_n>0$. We denote by $I_n$ the smallest value of $i \geq 0$ for which this condition is satisfied, and by $J_{I_n}$ the corresponding value of $j \geq 1$, so that $\compn{n}_{I_n,J_{I_n}}$ is the last computed quantity. 
	The starting point for the next time subinterval is then given by
	\begin{equation}\label{eq:nextinitsol}
		\initsol_{n}:=\recomp_{\delta_n} \cfpmap_n\compn{n}_{I_n,J_{I_n}}(t_n), 
	\end{equation}
	when Gauss--Legendre quadrature is used.  For Gauss--Lobatto points, the final collocation point coincides with $t_n$ and thus, in principle, no additional evaluation of the fixed-point map $\Phi_n$ is needed. Instead, one can simply apply a final recompression, that is, replace \eqref{eq:nextinitsol} in line \ref{line:recompr} of Algorithm~\ref{alg:basic} by 
	\begin{equation}\label{eq:nextinitsol_lobatto}
	  \initsol_{n}:=\recomp_{\delta_n} \compn{n}_{I_n,J_{I_n}}(t_n).
	  \end{equation}

	\begin{algorithm}
		\caption{Time integrator by thresholded fixed point iteration}\label{alg:basic}
		\begin{algorithmic}[1]
			\Require Initial value $\initsol_0$, final time $T$, number of time steps $N$, fixed point tolerances $\{\varepsilon_n\}_{n=1}^N$, recompression tolerances $\{\delta_n\}_{n=1}^N$, threshold decaying factor $\theta\in(0,1)$, $\nu\in(0,1)$, dimension $d$.
			\For{$n=1,2,\dots,N$}\Comment{time stepping} 
			\State Determine the collocation points $\tau_{1,n},\dots,\tau_{Q,n}$
			\State Initialize $i\leftarrow0$, $\compn{n}_{0,0}=0$, $\alpha_{0,n} = \frac{\| \tilde u_{n-1}\| }{2d-3}$, $\mathrm{res}\leftarrow \norm{\fpmap_n(0)}_Q$
			\While{$\mathrm{res} > \varepsilon_n$} \Comment{outer loop}
			\State Initialize $j\leftarrow0$
			\State Evaluate $\fpmap_n\compn{n}_{i,0}$ as in \eqref{eq:fixed_point_map}
			\Repeat \Comment{inner loop}
			\State Apply soft thresholding $\compn{n}_{i, j+1} = \softt_{\alpha_{i,n}}\fpmap_n\compn{n}_{i,j}$
			\State Update $j = j+1$
			\State Evaluate $\fpmap_n\compn{n}_{i,j}$
			\Until{$\normJ{\compn{n}_{i,j}-\compn{n}_{i,j-1}}\le\nu\frac{1-\rho}{\rho(1+\rho)}\normJ{\fpmap_n\compn{n}_{i,j}-\compn{n}_{i,j}}$}
			\State Set $J_i\leftarrow j$
			\State Update $\mathrm{res}\leftarrow\norm{\fpmap_n\compn{n}_{i,J_i} - \compn{n}_{i,J_i}}_Q$
			\State Set $\compn{n}_{i+1,0} = \compn{n}_{i, J_i}$
			\State Update $\alpha_{i+1,n}\leftarrow\theta\alpha_{i,n}$ \Comment{decrease $\alpha$ once inner loop finished}
			\State Update $i\leftarrow i+1$
			\EndWhile
			\State Set $I_n \leftarrow i-1$
			\State $\initsol_{n} = \recomp_{\delta_n} \cfpmap_n\compn{n}_{I_n,J_{I_n}}(t_n)$ \Comment{evaluate starting value for next subinterval} \label{line:recompr}
			\EndFor
			\State \Return $\initsol_1$, \dots, $\initsol_N$
		\end{algorithmic}
	\end{algorithm}
	
	In the remainder of the paper, we restrict our analysis mainly to the case of Gauss-Legendre nodes and only comment on the differences to Gauss-Lobatto nodes as an alternative choice.
	
	\subsection{Accuracy bounds}\label{sec:accuracy_bounds}
	As described in \Cref{subs:fixed_point_form}, the fixed point of the contraction $\fpmap_{\initsol_0}$ introduced in \eqref{eq:fixed_point_map} corresponds to the solution of an implicit Runge-Kutta method. In this section, we further exploit this equivalence property to prove some useful local accuracy bounds as in \cite{BDS:25}. We remark that these results only apply to the fixed point solution and not to the intermediate quantities produced by the fixed point method.
	
	We restrict our analysis to the case where $G_t$ is a skew-Hermitian operator at all times, in the sense that $\inp{G_t u}{v} = -\inp{u}{G_t v}$
	for all $u,v\in L^2(\Omega)$ and times $t \geq 0$. Skew-Hermitian operators preserve the norm of the exact solution of \eqref{eq:ivp}: namely, it holds that
	\begin{equation*}
		\frac{d\norm{u(t)}^2}{dt}=\frac{d}{dt}\left<u(t),u(t)\right> = \left<G_t u(t), u(t)\right> + \left<u(t), G_t u(t) \right> = 0.
	\end{equation*}
	In our case of interest, the Schr\"odinger equation \eqref{eq:schroedinger} as well as the formulation \eqref{eq:twisted_f} satisfy this requirement whenever the Hamiltonian $H_t = -\Delta + V_t$ is Hermitian.
	
	We first show that, in the case of norm-preserving operators, some choices of the collocation points can preserve the norm of the numerical solution.
	
	\begin{lemma}[Isometry preservation]\label[lemma]{lem:iso_pres}
		Let $v_0,w_0\in L^2(\Omega)$, and $\fixedp=\fpmap_{v_0}(\fixedp)$, $\widehat{\mathbf{w}} = \fpmap_{w_0}(\widehat{\mathbf{w}})$. Set $v_1:=\cfpmap_{v_0}(\fixedp)(t_1)$ and $w_1:=\cfpmap_{w_0}(\widehat{\mathbf{w}})(t_1)$. If Gauss-Legendre nodes are used, then
		\begin{equation*}
			\norm{v_1-w_1} = \norm{v_0-w_0}.
		\end{equation*}
		If Gauss-Lobatto nodes are used, the same property holds provided $G_t=G$ does not explicitly depend on time.
	\end{lemma}
	
	\begin{proof}
		The first part of the proof is analogous to \cite[Lemma 15]{BDS:25}, and we report it here to show that it also applies to a generic skew-Hermitian operator $G_t$. Observe first that, if $G_t$ is skew-Hermitian, then $\Re\left\langle G_t(u),u\right\rangle=0$ for all $u\in L^2(\Omega)$ at all times.
		
		Let $p:=\cfpmap_{v_0}(\fixedp)-\cfpmap_{w_0}(\widehat{\mathbf{w}})$. By linearity, we know that $p$ is a polynomial of degree $\nqp$ satisfying $\partial_tp(\tau_j) = G_{\tau_j}(p(\tau_j))$ for all $j=1, 2, \dots,\nqp$. It then holds that
		\begin{align*}
			\norm{v_1-w_1}^2-\norm{v_0-w_0}^2 &= \norm{p(t_1)}^2-\norm{p(t_0)}^2 =  \int_{t_0}^{t_1}2\Re\inp{\partial_tp(t)}{p(t)}\,dt \\&=2\Re\sum_{j=1}^\nqp\inp{\partial_tp(\tau_j)}{p(\tau_j)}\omega_j = 2\sum_{j=1}^\nqp\Re\left(\inp{G_{\tau_j}p(\tau_j)}{p(\tau_j)}\right)\omega_j = 0,
		\end{align*}
		where $\omega_j$, $j=1,\ldots, Q$, as in \eqref{eq: next_int_eval}.
		In the fourth equality we used the fact that $s\mapsto\inp{\partial_tp(s)}{p(s)}$ is a polynomial of degree $2\nqp-1$, and it can thus be integrated exactly by a Gauss-Legendre quadrature formula with $\nqp$ nodes.
		
		The degree of accuracy of Gauss-Lobatto quadrature formulas with $\nqp$ nodes is only $2\nqp-3$, and we can not apply the same argument. Instead, we assume that $G_t = G$  with time-independent $G$, and we observe that, since $G$ is bounded and skew-Hermitian, its spectrum $\sigma(G)$ is purely imaginary and there exists a projection-valued measure $\pi^G$ on $\R$ associated to the self-adjoint operator $-iG$ such that $G = \int_{-i\sigma(G)} i\lambda \,d\pi^G(\lambda)$. When applied to \eqref{eq:ivp} with initial condition $u(t_0)=u_0$, a generic Runge-Kutta method produces an approximation $u_1$ at time $t_0+h$ given by $u_1=  \int_{-i\sigma(G)} R(i\lambda h) \,d\pi^G(\lambda)  u_0$.
		
		Since Lobatto IIIA is a symmetric Runge-Kutta method, its stability function $R:\mathbb{C}\to\mathbb{C}$ satisfies $R(z)R(-z)=1$ for all $z\in\mathbb{C}$. In particular, this implies that $\lvert R(\lambda h)\rvert=1$ for $\lambda\in\sigma(G)$, since $\sigma(G)$ is purely imaginary. Then, since we have
		\begin{equation*}
			v_1-w_1= \int_{-i\sigma(G)} R(i\lambda h) \,d\pi^G(\lambda) (v_0-w_0)
		\end{equation*}
		by linearity and the isometry property of the spectral functional calculus, we obtain $\norm{v_1-w_1}^2 = \norm{v_0-w_0}^2$.
	\end{proof}
	
	Note that in the twisted variable formulation \eqref{eq:twisted_f}, the right-hand side $f_t$ becomes time-dependent even when starting from a time-independent Hamiltonian. Therefore, in this setting, we do not have exact isometry preservation when using Lobatto collocation points.

    For the local error analysis, we assume that the collocation solutions have sufficient temporal regularity. More precisely, for the $\nqp$-stage methods considered below, $(\partial_t-G_t)^m v(t)$ with $0\leq m\leq 2\nqp+1$
    are assumed to be uniformly bounded in $L^2(\Omega)$ for the exact and collocation solutions appearing in the analysis.
	\begin{lemma}[Local error at subinterval boundary]\label[lemma]{lem:local_bound}
		Let $\widehat{\mathbf{w}}=\fpmap_{u(t_0)}(\widehat{\mathbf{w}})$ be the collocation solution with Gauss-Legendre points associated to the exact initial value $u(t_0)$. Then
		\begin{equation*}
			\norm{\cfpmap_{u(t_0)}(\widehat{\mathbf{w}})(t_1) - u(t_1)} \leq \kappa_{2\nqp}h^{2\nqp+1},  
		\end{equation*}
 where the constant $\kappa_{2\nqp}$ is given by
		\begin{equation*}
			\kappa_{2\nqp} =  \frac{(Q!)^4}{(2Q+1)![(2Q)!]^2} \sup_{s\in[t_0,t_1]}\norm{(\partial_s-G_s)^{2\nqp + 1}\cfpmap_{u(t_0)}(\widehat{\mathbf{w}})(s)}.
		\end{equation*}
	\end{lemma}
	\begin{proof}
		Define $p(t) = \cfpmap_{u(t_0)}(\widehat{\mathbf{w}})(t)$ and $e(t)=p(t)-u(t)$. Then $e$ solves $\partial_te(t) = G_t e(t) + (\partial_t-G_t)p(t)$ with $e(t_0) = 0$.
		The variation of constants formula, see for example \cite[Theorem 11.2]{hairer_solving_1993}, yields
		\begin{equation*}
			e(t_1) = \int_{t_0}^{t_1}\varphi_{s\to t_1}(\partial_s - G_s)p(s)\,ds,
		\end{equation*}
		where $\varphi_{s\to t_1}$ is the exact flow of \eqref{eq:ivp} in $(s,t_1)$.
		Moreover, $(\partial_t - G_{\tau_j})p(\tau_j) = 0$ for all $j=1,\dots,\nqp$ because $p$ is a collocation solution. Therefore,
		\begin{equation*}
			\norm{p(t_1)-u(t_1)} = \norm{e(t_1)} = \biggnorm{\int_{t_0}^{t_1}\varphi_{s\to t_1}(\partial_s-G_s)p(s)\,ds - \sum_{q=1}^\nqp \omega_q \varphi_{\tau_q\to t_1}(\partial_t-G_{\tau_q})p(\tau_q)},
		\end{equation*}
		with $\omega_q$, $q=1,\ldots, Q$, as in \eqref{eq: next_int_eval}.
		Applying standard accuracy results for numerical quadrature (see \cite[Eq.~25.4.29]{AS65} and the Hilbert space estimate in \cite[Lemma~32]{BDS:25}), we obtain
		\begin{equation*}
			\norm{p(t_1)-u(t_1)}\leq C(\nqp)\sup_{s\in[t_0,t_1]}\bignorm{\partial_t^{2\nqp}\left(\varphi_{s\to t_1}(\partial_s - G_s)\cfpmap_{u(t_0)}(\widehat{\mathbf{w}})(s)\right)}.
		\end{equation*}
		
		The result then follows from the identity $\partial_t\varphi_{t\to t_1} g = - \varphi_{t\to t_1}(G_t(g))$
		for all $g\in L^2(\Omega)$. Indeed,
		\begin{equation*}
			\partial_t\varphi_{t\to t_1}g = \lim_{h\to 0}\frac{\varphi_{t+h\to t_1}g-\varphi_{t\to t_1}g}{h}=\lim_{h\to 0}\frac{\varphi_{t+h\to t_1}(g-\varphi_{t\to t+h}g)}{h}
		\end{equation*}
		and $\varphi_{t\to t+h}g = g + h\partial_s\varphi_{t\to s}g\big|_{s=t} + o(h)$
		with $\partial_s\varphi_{t\to s}g\big|_{s=t}  = G_s(\varphi_{t\to s}g)\big|_{s=t} = f_t(g)$,
		which gives the asserted identity. We also use that
		$\norm{\varphi_{t\to t_1}(g)}=\norm{g}$ for all $t$ and all
		$g\in L^2(\Omega)$, since \eqref{eq:ivp} preserves the norm of the
		solution.
	\end{proof}
	
	A result analogous to the above can also be obtained in particular for Gauss-Lobatto points, but with $h^{2Q+1}$ replaced by $h^{2Q-1}$ and a modified constant.
	The following two lemmas are direct adaptations of \cite[Lemma 17 and Proposition 13]{BDS:25} to our framework, and therefore we omit their proofs here.
	
	\begin{lemma}[Local error bound inside subintervals]\label[lemma]{lem:local_inside}
		Let $\widehat{\mathbf{v}}=\fpmap_{v_0}(\widehat{\mathbf{v}})$ be the collocation solution with Gauss-Legendre points  associated to the initial value $v_0$. It then holds that
		\begin{equation*}
			\norm{\cfpmap_{v_0}(\fixedp)(t)-u(t)} \leq \frac{\kappa_\nqp h^{\nqp+1}+\norm{v_0-u(t_0)}}{1-\rho}
		\end{equation*}
		for all $t\in(t_0,t_1)$, where
		\begin{equation*}
			\kappa_\nqp = \frac{\nqp!}{\sqrt{2\nqp+1}(2\nqp)!}\normWg{\partial_t^{\nqp+1} u}.
		\end{equation*}
	\end{lemma}

	Note that with appropriate modifications in the $Q$-dependent prefactor in $\kappa_Q$, \Cref{lem:local_inside} holds for any choice of the quadrature points. The following result, however, depends on the isometry preservation provided in our setting by the Gauss-Legendre points.  %
	\begin{prop}[Local error bound at subinterval boundaries]\label[prop]{prop:error_bound_1}
		The scheme with Gauss-Legendre points satisfies the error bound
		\begin{equation*}
			\norm{u(t_1)-\initsol_1} \leq \norm{u(t_0)-\initsol_0} + \frac{\rho\varepsilon}{1-\rho} + \kappa_{2\nqp}h^{2\nqp+1} + \delta.
		\end{equation*}
	\end{prop}

	\begin{remark}
		If the time integration scheme is not isometry-preserving, a similar result can be obtained under the weaker assumption of Lipschitz continuity. In this case, the statement of \Cref{lem:iso_pres} is replaced by 
		$\norm{\cfpmap_{u(t_0)}\widehat{\mathbf{w}}(t_1)-\cfpmap_{\initsol_0}\fixedp(t_1)} \leq (1+hL)\norm{u(t_0)-\initsol_{0}}$
		for some positive constant $L$, and the same arguments as in the proof of \Cref{prop:error_bound_1} yield the slightly modified bound
		\begin{equation*}
			\norm{u(t_1)-\initsol_1} \leq (1+hL)\norm{u(t_0)-\initsol_0} + \frac{\rho\varepsilon}{1-\rho} + \kappa_{2\nqp}h^{2\nqp+1} + \delta.
		\end{equation*}
	\end{remark}
	
	\section{Rank estimates with respect to discretizations}\label{sec:ranksdiscr}

	Our main goal is to show that the approximations produced by our method, that is, the iterates $\compn{n}_{i,j}$ and the approximations $\initsol_{n}$ at the beginning of each subinterval, have ranks that remain comparable to corresponding ranks of best approximations of the exact solution $u$. However, as mentioned in the introduction, different reference quantities are possible in the present context: while it is natural to compare the ranks $\initsol_{n}$ to those of $u(t_n)$, that is, to the exact solution at the same time instant, this is not the only reasonable possibility. Indeed, as the best possible evolution of ranks of the iterative scheme on each subinterval is rather determined by the ranks of the respective fixed points $\fixedpn{n}$, one could also be interested in taking the latter as reference ranks.
	
	In this section, we thus want to relate the approximation ranks of $\initsol_{n}$ to best approximation ranks of $\hat u_n := (\cfpmap_n\fixedpn{n})(t_n)$, that is, to the ranks of the fixed-point solution in the current time interval evaluated at the endpoint $t_n$. A schematic representation of the quantities involved in the analysis is presented in \Cref{fig:two_time_steps} for the first two time steps. Note that $\hat u_n$ represents the exact result of one step of \eqref{eq: next_int_eval} with starting value $\tilde u_{n-1}$, that is, including all accumulated errors from previous time steps. One limitation of this type of result is that these reference quantities depend on the choice of discretization and on errors from previous subintervals; for this reason, we will consider rank bounds in terms of the exact solution $u$ in \Cref{sec:comparison_wrt_exact}. 
	
	\begin{figure}[h]
	\centering
		\begin{tikzpicture}[scale=1,
			>=Latex,
			font=\small,
			point/.style={circle, fill=black, inner sep=1.2pt},
			exact/.style={very thick, blue},
			scheme/.style={very thick, green!70!black},
			fp/.style={thick, red, dashed},
			aux/.style={gray!60, thin, dashed},
			]

			\pgfmathsetmacro{\xh}{10}       %
			\pgfmathsetmacro{\xhh}{20}      %
			\pgfmathsetmacro{\xmin}{-1}
			\pgfmathsetmacro{\xmax}{\xhh+2.5}
			\pgfmathsetmacro{\xrange}{\xmax-\xmin}

			\pgfmathsetlengthmacro{\xunit}{0.95*\linewidth/\xrange}
			\begin{scope}[x=\xunit, y=1.0cm] %

				\draw[->] (\xmin,0) -- (\xmax,0) node[below right] {$t$};
				\foreach \x/\lab in {0/0, \xh/h, \xhh/2h}{
					\draw (\x,0.12) -- (\x,-0.12) node[below=2pt] {$\lab$};
					\draw[aux] (\x,0) -- (\x,3.7);
				}

				\coordinate (u0)   at (0,0.6);
				
				\coordinate (uh)   at (\xh,3.2);
				\coordinate (fph)  at (\xh,2.2);
				\coordinate (tuh)  at (\xh,1.0);
				
				\coordinate (u2h)  at (\xhh,2.8);
				\coordinate (fp2h) at (\xhh,1.9);
				\coordinate (tu2h) at (\xhh,0.8);

				\draw[exact,->]
				(u0) .. controls +(4.0,2.8) and +(-4.0,1.0) .. (uh);
				\draw[exact,->]
				(uh) .. controls +(4.0,0.3) and +(-4.0,0.6) .. (u2h);
				
				\draw[fp,->]
				(u0) .. controls +(4.0,2.2) and +(-4.0,0.7) .. (fph);
				\draw[fp,->]
				(tuh) .. controls +(4.0,2.0) and +(-4.0,0.6) .. (fp2h);
				
				\draw[scheme,->]
				(u0) .. controls +(4.0,0.15) and +(-4.0,0.15) .. (tuh);
				\draw[scheme,->]
				(tuh) .. controls +(4.0,-0.90) and +(-4.0,-0.30) .. (tu2h);

				\node[point] at (u0) {};
				\node[point] at (uh) {};
				\node[point] at (fph) {};
				\node[point] at (tuh) {};
				\node[point] at (u2h) {};
				\node[point] at (fp2h) {};
				\node[point] at (tu2h) {};

				\node[anchor=east] at ($(u0)+(-0.3,0)$) {$u_0$};
				
				\node[anchor=south] at ($(uh)+(0,0.15)$) {$u(h)$};
				\node[anchor=west, align=left, text width=0.23\linewidth, xshift=1mm] at (fph)
				{$\hat u_1$};
				\node[anchor=west, align=left, text width=0.23\linewidth, xshift=2.2mm] at (tuh)
				{$\initsol_1$};
				
				\node[anchor=south] at ($(u2h)+(0,0.15)$) {$u(2h)$};
				\node[anchor=west, align=left, text width=0.23\linewidth, xshift=1mm] at (fp2h)
				{$\hat u_2$};
				\node[anchor=west, align=left, text width=0.23\linewidth, xshift=1mm] at (tu2h)
				{$\initsol_2$};

				\matrix[matrix of nodes, anchor=south, column sep=3mm, row sep=1mm]
				at ({0.5*\xhh},3.85)
				{
					\tikz{\draw[exact] (0.2,0)--(1,0);}  & exact solution $u(t_n)$ &
					\tikz{\draw[scheme] (0.2,0)--(1,0);} & approximation $\initsol_n$ &
					\tikz{\draw[fp] (0.2,0)--(1,0);} & fixed point $\hat u_n$ \\
				};
				
			\end{scope}
		\end{tikzpicture}
		\caption{\small Schematic comparison over two time steps starting from the initial condition $u_0$.}\label{fig:two_time_steps}
	\end{figure}
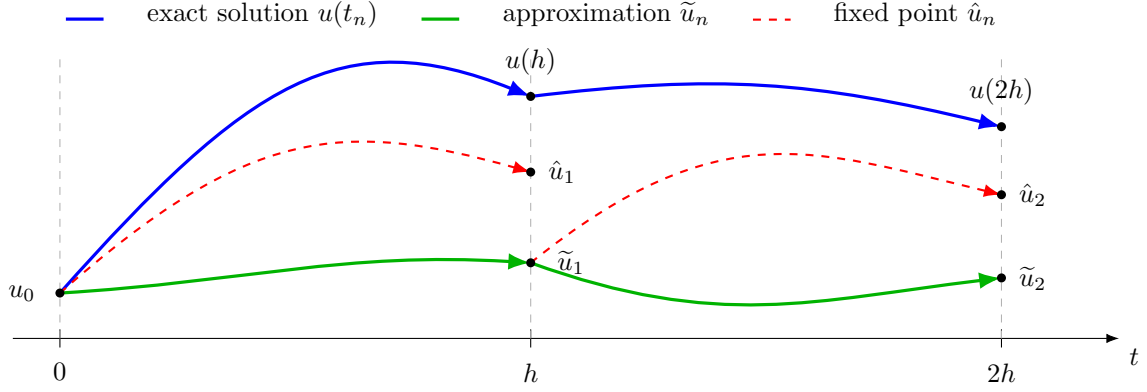

	In the following, we define the rank of an element $\mathbf{v}$ of $L^2(\Omega)^\nqp$ as 
	\begin{equation*}
		\rank_\mat(\mathbf{v}) = \max_{j=1,\dots,\nqp}\rank_\mat(v_j) \qquad \mat=1,2,\dots,E.
	\end{equation*}
	We also define the reference quantity
	\begin{equation} \label{eq: r best}
		r_{\mathrm{best}}(v, \eta) = \min \bigl\{r \in \N_0 \colon \bigl(\exists w \in L^2(\Omega), \max_{\mu=1,\ldots,E} \rank_\mu(w) \leq r \colon \norm{v - w} \leq \eta \bigr) \bigr\} 
	\end{equation}
	as the (uniform) best approximation ranks for an element $v \in L^2(\Omega)$ given the tolerance $\eta \geq 0$.

\subsection{Soft thresholding}

	We start by comparing the ranks of the iterates $\compn{n}_{i,j}$ to those of approximations of the fixed point $\fixedpn{n}$. Therefore, we consider a fixed subinterval and suppress the dependence on $n$ in this subsection to simplify notation.

	We first bound the difference between the quantities $\comp_{i,j}$ produced by the scheme and $\fixedp$ in terms of the distance between $\fixedp$ and the modified fixed points $\modfp{\alpha_i}$ associated with the soft thresholding.
	
	\begin{lemma}\label[lemma]{lem: bound iterates fixedp}
		For $i=1,2, \dots,I$ and $j=0, 1,  \dots, J_i$, 
		\begin{equation*}
			\normJ{\comp_{i,j}-\fixedp}\leq (1+\rho^j)\normJ{\fixedp-\modfp{\alpha_i}} + \frac{\rho^j}{1-\nu}\normJ{\fixedp-\modfp{\alpha_{i-1}}}.
		\end{equation*}
		For $i=0$, we have $J_0=1$ and $\normJ{\comp_{0,j}-\fixedp}\leq(1+\rho^j)\normJ{\fixedp-\modfp{\alpha_0}}$ for $j = 0, 1$.
	\end{lemma}
	
	\begin{proof}
		The proof is analogous to the derivation of equation (5.15) in \cite{BS17}. If $i=0$, then $\comp_{0,0}=0$ and $\comp_{0,1}=0$ because of the particular choice of $\alpha_0$. This entails that $J_0=1$ and that $\modfp{\alpha_0}=0$, so that $\normJ{\comp_{0,j}-\fixedp}\leq(1+\rho^j)\normJ{\fixedp-\modfp{\alpha_0}}$ for all $j=0,1$ and the statement holds.
		
		Next, assume that $i\geq1$. Using the triangle inequality together with \eqref{eq: iter def}, we obtain
		\begin{equation*}
			\normJ{\comp_{i,j}-\fixedp}\leq\normJ{\comp_{i,j}-\modfp{\alpha_i}}+\normJ{\modfp{\alpha_i}-\fixedp}=\normJ{\softt_{\alpha_i}\fpmap\comp_{i,j-1}-\softt_{\alpha_i}\fpmap\modfp{\alpha_i}}+\normJ{\modfp{\alpha_i}-\fixedp}.
		\end{equation*}
		The first term can be then further bounded by 
		\begin{equation*}
			\normJ{\softt_{\alpha_i}\fpmap\comp_{i,j-1}-\softt_{\alpha_i}\fpmap\modfp{\alpha_i}} \leq \rho^j\normJ{\comp_{i,0}-\modfp{\alpha_i}} \leq \rho^j \normJ{\comp_{i-1,J_{i-1}}-\fixedp} + \rho^j \normJ{\modfp{\alpha_i}-\fixedp},
		\end{equation*}
		which then yields
		\begin{equation}
			\normJ{\comp_{i,j}-\fixedp} \leq \rho^j\normJ{\comp_{i-1,J_{i-1}}-\fixedp} + (\rho^j+1)\normJ{\modfp{\alpha_i}-\fixedp}. \label{eq:lem1_1}
		\end{equation}

		Moreover, by definition of $\nu$ and \eqref{eq:thresholdcondition}, we obtain
		\begin{align*}
			\normJ{\comp_{i-1,J_{i-1}}-\modfp{\alpha_{i-1}}}&\leq\frac{\rho}{1-\rho}\normJ{\comp_{i-1,J_{i-1}}-\comp_{i-1,J_{i-1}-1}}\leq\frac{\nu}{1+\rho}\normJ{\fpmap\comp_{i-1,J_{i-1}}-\comp_{i-1,J_{i-1}}}\\&\leq\nu\normJ{\comp_{i-1,J_{i-1}}-\fixedp}
		\end{align*}
		and thus by triangle inequality also $\normJ{\comp_{i-1,J_{i-1}}-\fixedp} \leq \nu\normJ{\comp_{i-1,J_{i-1}}-\fixedp}+\normJ{\modfp{\alpha_{i-1}}-\fixedp}$.
		This entails $\normJ{\comp_{i-1,J_{i-1}}-\fixedp}\leq (1-\nu)^{-1}\normJ{\modfp{\alpha_{i-1}}-\fixedp}$,
		which yields the final result using \eqref{eq:lem1_1}.
	\end{proof}
	
	Next, we relate the difference between the fixed-point operator $\fpmap$ applied to iterates $\comp_{i,j}$ and the corresponding fixed-point solution $\fixedp$ to the associated soft thresholding error of the latter.

	\begin{lemma}\label[lemma]{lemma: Fv - hatv}
		For $i \geq 1$ and $j=0,1, \dots,J_i$, 
		\begin{equation*}
			\normJ{\fpmap\comp_{i,j}-\fixedp} \leq \frac{\rho(1+\rho^j)}{1-\rho}\normJ{\fixedp - \softt_{\alpha_i}\fixedp} + \frac{\rho^{j+1}}{(1-\nu)(1-\rho)}\normJ{\fixedp - \softt_{\alpha_{i-1}}\fixedp} \, .
		\end{equation*} 
		For $i=0$ and $j=0,1$, we have the same estimate with only the first summand on the right-hand side.
	\end{lemma}
	
	\begin{proof}
		Let $i \geq 1$. An application of Lemma \ref{lem: bound iterates fixedp} in the second step gives the bound
		\begin{equation}\label{eq:lem2_1}
			\normJ{\fpmap\comp_{i,j}-\fixedp}\leq\rho\normJ{\comp_{i,j}-\fixedp}  \leq \rho(1+\rho^j)\normJ{\fixedp-\modfp{\alpha_i}} + \frac{\rho^{j+1}}{1-\nu}\normJ{\fixedp-\modfp{\alpha_{i-1}}} \, .
		\end{equation}
		Inserting and subtracting the intermediate quantity $\softt_{\alpha_i}\fixedp$ yields 
		\begin{equation*}
			\normJ{\fixedp-\modfp{\alpha_i}} \leq \normJ{\fixedp -\softt_{\alpha_i}\fixedp}+\normJ{\softt_{\alpha_i}\fixedp-\modfp{\alpha_i}} \leq \normJ{\fixedp - \softt_{\alpha_i}\fixedp} + \rho\normJ{\fixedp-\modfp{\alpha_i}} \, ,
		\end{equation*}
		and thus
		\begin{equation}\label{eq:mod fp}
			\normJ{\fixedp-\modfp{\alpha_i}}\leq\frac{\normJ{\fixedp - \softt_{\alpha_i}\fixedp}}{1-\rho} \, .
		\end{equation}
		An analogous argument can also be applied to bound $\normJ{\fixedp-\modfp{\alpha_{i-1}}}$. Inserting this into \eqref{eq:lem2_1} then yields the statement of the lemma for $i \geq 1$. The case $i=0$ can be treated analogously. 
	\end{proof} 
	
	We are now in the position to prove the following local rank result for all computed iterates of the fixed-point iteration.
	
	\begin{prop}  \label{prop: ranks iterates fixed p}
		For $i \geq 1$, $j=1, 2, \dots,J_i$ and $\mat=1, 2, \dots,E$, 
		\begin{equation*}
			\rank_\mat{(\comp_{i,j})} \leq \frac{32}{(1-\rho)^2}\frac{\normJ{\fixedp - \softt_{\alpha_i}\fixedp}^2}{\alpha_i^2} + \frac{8/(1-\nu)^2}{\theta^2(1-\rho)^2}\frac{\normJ{\fixedp-\softt_{\alpha_{i-1}}\fixedp}^2}{\alpha_{i-1}^2}+\rank_\mat(\softt_{\mat,\alpha_i/2}\fixedp) \, ,
		\end{equation*}
		and the same bound applies to $\rank_\mat{(\comp_{i+1,0})}$ for $ i \geq 1$;
		moreover, $\comp_{0,0} = \comp_{0,1} = \comp_{1,0} =0$.
	\end{prop}
	
	\begin{proof}
		We first assume $i \geq 1$, $j = 1,2, \dots, J_i$ and use \cite[Lemma 4.3]{BS17} to obtain the bound
		\begin{equation*}
			\rank_\mat(\comp_{i,j})=\rank_\mat(\softt_{\alpha_i}\fpmap\comp_{i,j-1}) \leq 4\frac{\normJ{\fixedp - \fpmap\comp_{i,j-1}}^2}{\alpha_i^2} + \rank_\mat(\softt_{\mu,\alpha_i/2} \fixedp),
		\end{equation*}
		and then apply Lemma \ref{lemma: Fv - hatv} to $\normJ{\fixedp - \fpmap\comp_{i,j-1}}$. 
		If $i = 1$ and $j = 0$, then using that $J_0 = 1$ we obtain $\rank_\mat(\comp_{1,0}) = \rank_\mat(\comp_{0,1}) = 0$,
		since $\comp_{0,0} = \comp_{0,1} = 0$ and the corresponding ranks are thus all equal to zero.
		For $i \geq 2$ and $j=0$ we can use that $\rank_\mat(\comp_{i,0})= \rank_\mat(\comp_{i-1,J_{i-1}}) = \rank_\mat(\softt_{\alpha_{i-1}}\fpmap\comp_{i-1,J_{i-1}-1})$ and argue analogously to the first case.
	\end{proof}
	
	Note that the above proposition holds for all time subintervals without any restrictions on the accuracy of the scheme. Therefore, using the appropriate truncation tolerances and fixed points $\fixedpn{n}$, it can also be extended without further restrictions to hold globally.

	In order to relate the obtained bounds to the quasi-optimal ones, however, one needs to assume a certain decay property of the singular values of the individual matricizations.

	\subsection{Algebraic and exponential decay of singular values}\label[subsection]{subsec:sv_decay}
	
	To interpret the rank bounds from Proposition~\ref{prop: ranks iterates fixed p}, we relate them to the intrinsic compressibility of the fixed point $\fixedp$. 
	
	For each matricization $\mat = 1, \dots,E$ and $q=1,\dots,\nqp$, recall that $(\sigma_{\mat,k}(\hat v_q))_{k\ge1}$ denote the singular values of $\fixedp=(\hat v_q)_{q=1,\dots,Q}$. Their decay determines the rate of best low-rank approximation and, equivalently, the growth of the minimal rank required to reach a prescribed accuracy $\varepsilon$. Under the assumption of algebraically or exponentially decaying singular values, we obtain quasi-optimal rank bounds for the computed iterates.
		
	\begin{cor}[Quasi-optimal ranks] \label{cor:quasioptlocal}
	Let $\fixedp = (\hat v_q)_{q = 1, \dots, \nqp}$. We then have the following bounds on ranks of iterates for the given time subinterval.
		\begin{enumerate}[\rm(i)]
			\item If $\sigma_\mat(\hat v_q)\in\ell^{p,\infty}$ for all $q=1,2,\dots, \nqp$ and some $p\in(0,2)$, for all $\mat=1, 2, \dots,E$, then
			\begin{equation*} 
				\normJ{\comp_{i,j}-\fixedp}\leq k_1EM(\fixedp)^{\frac{p}{2}}\alpha_i^{1-\frac{p}{2}}, \qquad
				\rank_\mat(\comp_{i,j}) \leq k_2(1+E^2)M(\fixedp)^p\alpha_i^{-p},
			\end{equation*}
			where $M(\fixedp):=\max_\mat \max_q\lvert\sigma_\mat(\hat v_q)\rvert_{\ell^{p,\infty}}$ and the constants $k_1,k_2$ depend only on $p$, $\rho$, $\nu$ and $\theta$. In particular, setting $\lambda_{i,j}:=M(\fixedp)^{\frac{p}{2}}\alpha_i^{1-\frac{p}{2}}$ we have
			\begin{equation*}
				\normJ{\comp_{i,j}-\fixedp}\lesssim E\lambda_{i,j}, \qquad \max_{\mat=1,\dots,E}\rank_\mat(\comp_{i,j}) \lesssim E^2 M(\fixedp)^{\frac{1}{s}}\lambda_{i,j}^{-\frac{1}{s}}, \text{ where } s = \frac{1}{p}-\frac12.
			\end{equation*}
			\item If $\sigma_{\mat,j}(\hat v_q)\leq Ce^{-cj^\beta}$ for all $q = 1, 2, \dots, Q$, $j \in \mathbb N$ and $\mat=1, 2,  \dots,E$, then
			\begin{equation*}
				\normJ{\comp_{i,j}-\fixedp}\leq k_3E(1+i\lvert\log\theta\rvert)^{\frac{1}{2\beta}}\theta^i, \qquad
				\rank_\mat(\comp_{i,j}) \leq k_4(1+E^2)(1+i\lvert\log\theta\rvert)^{\frac{1}{\beta}},
			\end{equation*}
			where the constants $k_3$, $k_4$ depend only on $C, c, \beta, \alpha_0, \rho, \nu$ and $\theta$. In particular, setting $\lambda_{i,j}:=\tilde{\theta}^i$ for some $\tilde{\theta}\in(\theta,1)$ we have
			\begin{equation*}
				\normJ{\comp_{i,j}-\fixedp}\lesssim E\lambda_{i,j}, \qquad \max_{\mat=1,\dots,E}\rank_\mat(\comp_{i,j})\lesssim E^2(1+\lvert\log\lambda_{i,j}\rvert)^\frac{1}{\beta}.
			\end{equation*}
		\end{enumerate}
	\end{cor}
	
	The result follows immediately from \cite[Proposition 3.6]{BS17} combined with Lemma \ref{lem: bound iterates fixedp}, equation \eqref{eq:mod fp} and Proposition \ref{prop: ranks iterates fixed p}.
	What now remains to consider is a comparison of the obtained approximation ranks at the subinterval boundaries.
	
	\subsection{Recompression at interval endpoints and global bounds}

	Finally, in order to be able to relate the approximation ranks of $\initsol_{n}$ to best approximation ranks of $\hat u_n$, a recompression is performed in \eqref{eq:nextinitsol}. This approach is similar to the scheme in \cite{BDS:25}, but here the recompression tolerance is adjusted adaptively.
	The following lemma is a restatement of \cite[Lemma~2]{bachmayr_adaptive_2015} (see also \cite[Lemma~5.4]{Bachmayr23}).  It provides a relation to best approximation ranks based on rank truncations with sufficiently large tolerances. 
	
	\begin{lemma} \label{lem: ht best tolerance}
		Let $\kappa = \sqrt{E}$ and $\alpha>0$. Then for $u, v \in L^2(\Omega)$ such that $\norm{ u - v } \leq \eta$ and $\tilde v= \mathcal R_{\kappa (1 + \alpha)\eta} (v)$, we have $\norm{ u - \tilde v} \leq \bigl(1 + \kappa(1+\alpha) \bigr) \eta$ while
		\begin{equation*}
			\max_{\mu=1,\ldots,E} \rank_\mu(\tilde v) \leq r_{\mathrm{best}} (u, \alpha \eta) \, ,
		\end{equation*}
		where the best approximation ranks are defined as in \eqref{eq: r best}.
	\end{lemma}
	
	Using the previous lemma, we are directly in the position to prove $(i)$ rank bounds on the approximations $\initsol_{n}$ at the subinterval boundaries, and $(ii)$ a global error bound which scales only linearly in the final time $T$.
	
	\begin{theorem}
		Choosing the iteration and recompression tolerances proportional to the time discretization error of the Gauss method, that is, $\varepsilon_n = h^{2Q}$ and $\delta_n = \frac{2 \kappa \rho}{1- \rho} h^{2Q}$, with $\kappa=\sqrt{E}$, the computed approximations $\initsol_n$ satisfy
		\begin{equation*}
			\norm{\initsol_n-u(t_n)} \leq \eta_0 + \Big(\kappa_{2\nqp} + \frac{C_G \Lambda_Q(1 +  2 \kappa)}{1- \rho} \Big)h^{2\nqp}t_n ,
		\end{equation*}
		for all $1 \leq n \leq N$, and their ranks are bounded by
		\begin{equation*}
			\max_{\mu=1,\ldots,E} \rank_\mu(\tilde u_n) \leq r_{\mathrm{best}} \Big(\hat u_n, \frac{\rho \varepsilon_n }{1- \rho}  \Big) \, .
		\end{equation*}
	\end{theorem}
	
	Note that here, the magnitude of the recompression tolerance $\delta_n$ only scales with the square root of the dimension $d$ of the problem.
	
	\begin{proof}
		From \eqref{eq:lipsch} we first obtain
		\begin{equation*}
			\| \cfpmap_n\compn{n}_{I_n,J_{I_n}} -  \cfpmap_n\fixedpn{n} \|_{\cW(t_{n-1},t_n)} \leq \rho \normJ{\compn{n}_{I_n,J_{I_n}} - \fixedpn{n}} \, ,
		\end{equation*}
		where, using \eqref{eq:normQ_vs_normW},
		\begin{align*}
			\normJ{\compn{n}_{I_n,J_{I_n}} - \fixedpn{n}} &\leq \normJ{ \compn{n}_{I_n,J_{I_n}} - \fpmap_n \compn{n}_{I_n,J_{I_n}}} +  \normJ{\fpmap_n \compn{n}_{I_n,J_{I_n}} - \fixedpn{n}} \\
			&\leq \varepsilon_n +  \norm{\cfpmap_n\compn{n}_{I_n,J_{I_n}} - \cfpmap_n\fixedpn{n}}_{\cW(t_{n-1},t_n)} \, .
		\end{align*}
		Overall, this yields the bound 
		\begin{equation*}
			\| \cfpmap_n\compn{n}_{I_n,J_{I_n}}(t_n) -  \cfpmap_n\fixedpn{n}(t_n) \| \leq \norm{\cfpmap_n\compn{n}_{I_n,J_{I_n}} - \cfpmap_n\fixedpn{n}}_{\cW(t_{n-1},t_n)}\leq \frac{\rho \varepsilon_n }{1- \rho} \, .
		\end{equation*}
		Choosing $\delta_n = \frac{2 \kappa \rho \varepsilon_n }{1- \rho}$ and $\alpha = 1$ in Lemma \ref{lem: ht best tolerance} then gives us
		\begin{equation*}
			\max_{\mu=1,\ldots,E} \rank_\mu(\tilde u_n) \leq r_{\mathrm{best}} \Big(\cfpmap_n\fixedpn{n}(t_n), \frac{\rho \varepsilon_n }{1- \rho} \Big) \, .
		\end{equation*}
		Note that none of the computed rank bounds relies on any limitations regarding the choice of the iteration tolerances $\varepsilon_n$. 
		
		Finally, inductive application of \Cref{prop:error_bound_1} yields
		\begin{equation} 
			\norm{\initsol_n-u(t_n)} \leq \eta_0 +\kappa_{2\nqp}h^{2\nqp+1}n  + \frac{\rho (1 + 2 \kappa)}{1-\rho}\sum_{k=1}^n\varepsilon_k.
		\end{equation}
		Choosing $\varepsilon_k = h^{2Q}$, we can simplify the above expression to
		\begin{equation*}
			\norm{\initsol_n-u(t_n)} \leq \eta_0 + \Big(\kappa_{2\nqp} + \frac{C_G \Lambda_Q(1 +  2 \kappa)}{1- \rho} \Big)h^{2\nqp} t_n,
		\end{equation*}
		which scales only linearly in the time $t_n$.
	\end{proof}
	
	\begin{remark}[Rank control for Gauss-Lobatto nodes]
	At first sight, the direct use of stage endpoint values produced by the fixed-point iteration as in \eqref{eq:nextinitsol_lobatto} seems to be an advantage of Gauss-Lobatto nodes, since a possible rank increase due to the application by $\Phi_n$ is avoided and in principle, no further rank truncation would be required. However, due to the lack of isometry preservation for our formulation, we only obtain an error bound that is exponential rather than linear with respect to $T$. Moreover, the additional application of $\Phi_n$ in \eqref{eq:nextinitsol} also has the effect of reducing the error by a further factor of $\rho$, which plays as important role in the further analysis in the following section.
	\end{remark}
	
	Proposition~\ref{prop: ranks iterates fixed p} and Corollary~\ref{cor:quasioptlocal} show that, under algebraic or exponential decay of the matricization singular values of the fixed point $\fixedp$, the ranks of the computed iterates $\comp_{i,j}$ exhibit the same scaling in the target accuracy as the corresponding best-approximation ranks $r_{\mathrm{best}}(\fixedp,\cdot)$, up to constants depending only on $E, \rho,\nu,\theta$ as well as  the decay parameters. In particular, the dependence on the dimension enters explicitly only polynomially through $d^2$ (however, note that a further dependence on $d$ can arise from problem parameters and solution approximability).
	
	At the subinterval boundaries, the ranks of the recompressed approximations $\initsol_n$ are bounded by the best approximation ranks of $\hat u_n$ for a smaller tolerance $\tfrac{\rho}{1-\rho}\varepsilon_n$. Here the factor $\rho$, which is proportional to $h$ according to \eqref{eq:rhodef}, arises from the Lipschitz estimate for the evaluation map. This modification of the tolerance by a constant factor does not affect the resulting quasi-optimal approximation rates, that is,
	\begin{equation*}
		r_{\mathrm{best}}\bigl(\hat u_n , \tfrac{\rho}{1-\rho}\varepsilon_n\bigr) \lesssim (\rho \varepsilon_n)^{-\frac1s} \quad \text{and} \quad r_{\mathrm{best}}\bigl(\hat u_n, \tfrac{\rho}{1-\rho}\varepsilon_n\bigr) \lesssim \bigl(1 +\abs{\log \rho} +\abs{\log \varepsilon_n}\bigr)^{\frac1\beta} \, ,
	\end{equation*}
	if the singular value sequences $\sigma_\mu(\hat u_n)$ have algebraic or exponential decay as in Corollary \ref{cor:quasioptlocal}.

	\section{Rank estimates with respect to exact solutions}\label{sec:comparison_wrt_exact}
	
	The analysis outlined in the previous section shows that we can easily obtain local bounds in terms of the fixed point solution in each time subinterval.  While comparisons with the fixed points $\fixedpn{n}$ are natural for analyzing the algorithm on each subinterval, the ranks of $\fixedpn{n}$ and $\hat u_n$ are not purely problem-intrinsic: they depend, among other things, on the chosen spatial discretization, on the quadrature/collocation parameters, and on perturbations inherited from previous steps such as  recompression and iteration errors. In particular, spatial discretization errors may alter the singular value decay of the discrete solution, for example, by introducing oscillations or numerical noise, which can artificially increase or decrease the effective ranks needed to represent $\fixedpn{n}$.
	
	Therefore, comparing the obtained ranks also to those required for approximating the true solution $u(t_n)$  provides a discretization-independent benchmark and clarifies whether the method tracks the intrinsic low-rank structure of the underlying continuous problem rather than potential artifacts of the discretization. For this reason, in this section we set out to derive error and rank bounds on the computed approximations by taking the exact solution $u$ as a reference.
	
	\subsection{Challenges in generalizing the approach of \cite{BDS:25} to higher-order tensors}
	Our starting point is the analysis developed in \cite{BDS:25} for the case $d=2$. There, with $E = 1$ for $d=2$, the analysis crucially relies on the stability bound
	\begin{equation}\label{eq:stabilityE}
			\normJ{ \mathbf u - \mathbf \cS_{\alpha}\mathbf u} \leq  E \normJ{ \mathbf u - \mathbf \cS_{\beta}\mathbf u} \leq E^2 \frac{\beta}{\alpha} \normJ{ \mathbf u - \mathbf \cS_{\alpha}\mathbf u }
	\end{equation}
 	to compare soft-thresholding errors associated to different thresholds. This is shown in \cite[Lemma 10]{BDS:25} for $d=2$ and in Appendix \ref{app:ststab} for general $d$. While \eqref{eq:stabilityE} is sharp for $d=2$, numerical evidence suggests that this is not the case for $d>2$. It is not clear whether the explicit dependence on the dimension $d$ can be removed in both inequalities.
 
	A direct extension of the rank analysis from \cite{BDS:25} using \eqref{eq:stabilityE} as a high-dimensional substitute for \cite[Lemma 10]{BDS:25} unfortunately results in an error bound with respect to the exact solution scaling exponentially in the dimension as well as a higher-order polynomial dependence of the rank bounds on $E$. Additionally, we would like to improve the proportionality factor of the rank bounds from \cite[Theorem 27]{BDS:25}, which has the unfavorable scaling $\frac{1}{h^3}$ for small timesteps.

	These observations motivate the use of a different strategy to decrease the threshold $\alpha$ during the iterations as well as the development of a different argument for the rank control across subintervals, which we outline in this section. The global error and rank bounds presented in \Cref{prop:global_error_rank} do not exhibit an explicit dependence on $d$ (although an \emph{implicit} dependence may occur via other parameters) and the obtained global rank bound scales linearly in the number $N= \frac Th$ of subintervals. Moreover, we show that if the singular values of $u$ decay algebraically or exponentially, then the approximation ranks obtained can be bounded in terms of the optimal approximation ranks of $u$, up to a constant whose dependence on the dimension $d$ is only polynomial. This quasi-optimality result can be deduced from \Cref{prp:quasiopt} and \Cref{cor:quasiopt}.

	\subsection{Local error and rank bounds}
	
	We first derive local bounds on the ranks and the error associated to the intermediate quantities produced by our scheme. For this reason, we focus on a particular interval $\cI_n$ and suppress the dependence on $n$ for simplicity. Without restriction, we can also assume that the stopping index $I$ is greater than zero, since otherwise the produced iterate on the quadrature nodes would be equal to zero and thus the scheme \eqref{eq: next_int_eval} would just reproduce the initial data. In this case, the rank bounds become trivial.
	
	The following proposition is a key result of this subsection: it provides an estimate of the ranks of the intermediate quantities $\comp_{i,j}$ and of the error with respect to the exact solution $\exsol$ in terms of the singular value decay of $\exsol$.
	\begin{prop}\label[proposition]{prop:local_error_rank}
		Let $\normJ{\exsol-\fixedp}\leq \frac{1-\rho}{4(1+\rho)}\frac{1-\nu}{3-2\nu}\varepsilon$. Then, for all $i=1, 2, \dots,I$, for all $j=1,\dots,J_i$ and for all $\mat=1,2, \dots,E$, 
		\begin{align*}
			\rank_\mat{(\comp_{i,j})} &\leq \frac{128}{(1-\rho)^2}\frac{\normJ{\exsol-\softt_{\alpha_i}\exsol}^2}{\alpha_i^2} + \frac{32/(1-\nu)^2}{\theta^2(1-\rho)^2}\frac{\normJ{\exsol-\softt_{\alpha_{i-1}}\exsol}^2}{\alpha_{i-1}^2}+\rank_\mat(\softt_{\mat,\alpha_i/2}\exsol), \\
			\normJ{\comp_{i,j}-\exsol}&\leq\frac{(13-9\nu)/(3-2\nu)}{1-\rho}\left(\normJ{\exsol-\softt_{\alpha_i}\exsol} + \frac{1}{2(1-\nu)}\normJ{\exsol-\softt_{\alpha_{i-1}}\exsol}\right) \,.
		\end{align*}
		The above accuracy bound also holds for $j = 0$ and the same rank bound applies to $\rank_\mat(\comp_{i+1,0})$ for $i \geq 1$. Additionally, it holds that
		\begin{equation*}
			\normJ{\comp_{0,j}-\exsol}\leq\frac{(13-9\nu)/(3-2\nu)}{1-\rho} \normJ{\exsol-\softt_{\alpha_0}\exsol} \, ,
		\end{equation*}
		with $\comp_{1,0} = \comp_{0,1} = \comp_{0,0} = 0$.
		Moreover, the statement of Corollary \ref{cor:quasioptlocal} holds true with $\exsol$ in place of $\fixedp$, with modified constants depending on the same quantities.
	\end{prop}	
	We prove \Cref{prop:local_error_rank} by means of two auxiliary results.  In the first lemma, we derive a bound on the distance between the exact solution $\exsol$ and the intermediate quantities $\comp_{i,j}$ and $\fpmap\comp_{i,j}$.
	\begin{lemma}\label[lemma]{lem:lemma2}
		For $i=1,2,\dots,I$ and $j=0,1,\dots,J_i$, 
		\begin{equation*}
			\normJ{\fpmap\comp_{i,j}-\exsol} \leq C_j(\rho, \nu) \normJ{\exsol-\fixedp} + \rho D_j(\rho) \normJ{\exsol-\softt_{\alpha_i}\exsol} + \rho F_j (\rho, \nu) \normJ{\exsol-\softt_{\alpha_{i-1}}\exsol}
		\end{equation*}
		as well as
		\begin{equation*}
			\normJ{\comp_{i,j}-\exsol} \leq C_j^\prime(\rho, \nu) \normJ{\exsol-\fixedp} + D_j(\rho) \normJ{\exsol-\softt_{\alpha_i}\exsol} + F_j(\rho, \nu) \normJ{\exsol-\softt_{\alpha_{i-1}}\exsol},
		\end{equation*}
		where the constants are given by
		\[
		\begin{gathered}
			C_{j}(\rho, \nu) = \frac{(1-\nu)(1+\rho)+2\rho^{j+1}(2-\nu)}{(1-\nu)(1-\rho)}, \qquad C^\prime_{j}(\rho, \nu) = \frac{(1-\nu)(3-\rho)+2\rho^j(2-\nu)}
			{(1-\nu)(1-\rho)}, \\
			D_{j}(\rho) = \frac{1+\rho^j}{1-\rho}, \qquad F_{j}(\rho, \nu) = \frac{\rho^j}{(1-\nu)(1-\rho)}.
		\end{gathered}
		\]
	\end{lemma}
	
	\begin{proof}
		Using the triangle inequality  $\normJ{\fpmap\comp_{i,j}-\exsol} \leq \normJ{\fpmap\comp_{i,j}-\fixedp} + \normJ{\fixedp-\exsol}$,
		applying \Cref{lemma: Fv - hatv} to the first term on the right and noticing that
		\begin{equation}\label{eq:softt_fixedp_vs_exact}
			\normJ{\fixedp-\softt_{\alpha}\fixedp} \leq \normJ{\exsol-\softt_{\alpha}\exsol} + 2\normJ{\exsol-\fixedp},
		\end{equation}
		for any value of the threshold $\alpha$, we obtain the first result by rearranging terms. The bound on $\normJ{\comp_{i,j}-\exsol}$ can be obtained analogously, starting from \Cref{lem: bound iterates fixedp} and using \eqref{eq:mod fp}.
	\end{proof}
	
	Next we show that, as long as the stopping tolerance $\varepsilon$ is sufficiently large compared to $\normJ{\exsol-\fixedp}$, that is, the deviation of the fixed-point from the exact solution on the current interval, we can obtain error bounds that only depend on the thresholding error associated with $\exsol$.
	
	\begin{lemma}\label[lemma]{lem:lemma3}
		If $\normJ{\exsol-\fixedp}\leq \frac{1-\rho}{4(1+\rho)}\frac{1-\nu}{3-2\nu}\varepsilon$, then for $i=1,2, \dots,I$ and $j=0,1,\dots,J_i$,
		\begin{align*}
			&\normJ{\fpmap\comp_{i,j}-\exsol}\leq\frac{4}{1-\rho}\left(\normJ{\exsol-\softt_{\alpha_i}\exsol} + \frac{1}{2(1-\nu)}\normJ{\exsol-\softt_{\alpha_{i-1}}\exsol}\right), \\
			&\normJ{\comp_{i,j}-\exsol}\leq\frac{(13-9\nu)/(3-2\nu)}{1-\rho}\left(\normJ{\exsol-\softt_{\alpha_i}\exsol} + \frac{1}{2(1-\nu)}\normJ{\exsol-\softt_{\alpha_{i-1}}\exsol}\right).
		\end{align*}
	\end{lemma}
	
	\begin{proof}
		For a fixed $i=0,1,\dots,I$ with $I \geq 1$, choose $j$ so that $\comp_{i,j}$ does not satisfy the stopping criterion. %
		Since $\normJ{\comp_{i,j}-\fpmap\comp_{i,j}} \leq (1+\rho)\normJ{\comp_{i,j}-\fixedp}$, it then holds that
		\begin{align}
			\begin{split}\label{eq:bound_en_by_unvn}
				\varepsilon & \leq 2(1+\rho)\normJ{\fixedp-\modfp{\alpha_i}} + \frac{1+\rho}{1-\nu}\normJ{\fixedp-\modfp{\alpha_{i-1}}} \\ &\leq 2\frac{3-2\nu}{1-\nu}\frac{1+\rho}{1-\rho}\normJ{\exsol-\fixedp} + 2\frac{1+\rho}{1-\rho}\normJ{\exsol-\softt_{\alpha_i}\exsol} + \frac{1+\rho}{(1-\nu)(1-\rho)}\normJ{\exsol-\softt_{\alpha_{i-1}}\exsol},
			\end{split}
		\end{align}
		where we used \Cref{lem: bound iterates fixedp} together with the fact that $\rho^j\leq1$ for all $j\geq0$ in the third inequality, and \eqref{eq:mod fp} combined with \eqref{eq:softt_fixedp_vs_exact} in the fourth inequality. Since we can bound $\varepsilon$ from below in terms of $\normJ{\exsol-\fixedp}$, we find that
		\begin{equation*}
			\normJ{\exsol-\fixedp}\leq\frac{1-\nu}{3-2\nu}\normJ{\exsol-\softt_{\alpha_i}\exsol} + \frac{1}{2(3-2\nu)}\normJ{\exsol-\softt_{\alpha_{i-1}}\exsol},
		\end{equation*}
		for all $i=1,2, \dots, I$. Finally, inserting this into \Cref{lem:lemma2} we obtain
		\begin{align*}
			&\normJ{\fpmap\comp_{i,j}-\exsol}\leq\frac{(11-7\nu)\rho+1-\nu}{(1-\rho)(3-2\nu)}\left(\normJ{\exsol-\softt_{\alpha_i}\exsol} + \frac{1}{2(1-\nu)}\normJ{\exsol-\softt_{\alpha_{i-1}}\exsol}\right), \\
			&\normJ{\comp_{i,j}-\exsol}\leq\frac{(\nu-1)\rho+13-9\nu}{(1-\rho)(3-2\nu)}\left(\normJ{\exsol-\softt_{\alpha_i}\exsol} + \frac{1}{2(1-\nu)}\normJ{\exsol-\softt_{\alpha_{i-1}}\exsol}\right).
		\end{align*}
		The final statement follows from the fact that $(11-7\nu)\rho+1-\nu \leq 12-8\nu$ and $(\nu-1)\rho+13-9\nu \leq 13-9\nu$ for all $\rho\in(0,1)$.
	\end{proof}

	We are now in a position to prove \Cref{prop:local_error_rank}.
	\begin{proof}[Proof of \Cref{prop:local_error_rank}]
		The error bound is given by \Cref{lem:lemma3} for $i \geq 1$ and $j \geq 0$. In order to obtain the rank bound for $i \geq 1$, $j \geq 1$, we first apply \cite[Lemma 4.3]{BS17} to get
		\begin{equation*}
			\rank_\mat(\comp_{i,j})=\rank_\mat(\softt_{\alpha_i}\fpmap\comp_{i,j-1}) \leq 4\frac{\normJ{\exsol-\fpmap\comp_{i,j-1}}^2}{\alpha_i^2} + \rank_\mat(\softt_{t,\alpha_i/2}\exsol)
		\end{equation*}
		and then apply \Cref{lem:lemma3} to $\normJ{\exsol-\fpmap\comp_{i,j-1}}$. All remaining special cases can be treated analogously to the previous section.
	\end{proof}
	
	\subsection{Global error and rank bounds}
	We next show how the local bounds obtained in the previous section can be extended to global bounds in the whole time interval $[0,T]$, where we use the notation introduced in Section \ref{sec:globaliter}. We also define $\albar_n:=\alpha_{I_n,n}$ and $\altil_n:=\alpha_{I_n-1,n} = \albar_{n}/\theta$. Recall that here we still assume $I_n \geq 1$ without restriction.
	
    For the global error analysis below, we set $\delta_n=0$, so that no recompression is performed at the time-step endpoints. This choice allows us to track the propagation of errors and ranks without introducing additional recompression terms. In practice, taking $\delta_n>0$ may substantially reduce the endpoint ranks, at the expense of an additional approximation error. The following proposition provides an error bound at the discrete time instants $t_n$, for $n=1,2,\dots,N$, as well as a bound on the rank of the approximations $\initsol_n$.
	
	\begin{prop}\label[proposition]{prop:global_error_rank}
		For all $n=1,2,\dots,N$, set
		\begin{equation}\label{eq:epsilon_n}
			\varepsilon_n = \frac{4(1+\rho)}{(1-\rho)^2}\frac{3-2\nu}{1-\nu}\left[\kappa_\nqp h^{\nqp+1} + \kappa_{2\nqp}h^{2\nqp+1}(n-1) + \eta_0 + \frac{\rho}{1-\rho}\sum_{k=1}^{n-1}\varepsilon_k\right]. %
		\end{equation}
		With this choice, the scheme with Gauss-Legendre points achieves a global error bounded by
		\begin{equation*}
			\norm{\initsol_n-u(t_n)} \leq \left(\eta_0 + \kappa_\nqp h^{\nqp+1} + \frac{(1-\rho)^3}{4C_G\Lambda_\nqp(1+\rho)}\frac{1-\nu}{3-2\nu}\kappa_{2\nqp}h^{2\nqp}\right)\exp\left(\frac{4C_G\Lambda_Q(1+\rho)}{(1-\rho)^3}\frac{3-2\nu}{1-\nu}t_n\right).
		\end{equation*}
		Moreover, for all $\mat=1,2,\dots,E$, the ranks of the scheme are bounded by
		\begin{multline*}
			\rank_\mat(\initsol_n) \leq \rank_\mat(\initsol_0) + \nqp\rank_\mat(V)\sum_{k=1}^n\Bigg(\frac{128}{(1-\rho)^2}\frac{\normJ{\exsoln{k}-\softt_{\albar_k}\exsoln{k}}^2}{\albar_k^2}  \\ + \frac{32/(1-\nu)^2}{\theta^2(1-\rho)^2}\frac{\normJ{\exsoln{k}-\softt_{\altil_k}\exsoln{k}}^2}{\altil_k^2} + \rank_\mat\Big(\softt_{\mat,\albar_k/2}\exsoln{k}\Big)\Bigg).
		\end{multline*}
	\end{prop}
	\begin{remark}
		Equation \eqref{eq:epsilon_n} reveals that the value of the stopping tolerance $\varepsilon_n$ increases exponentially in $n$. This translates to the exponential dependence of the error bound on the final time. However, the exponent does not depend on the number of time steps nor on the dimension $d$, and in fact this error bound turns out to be rather pessimistic in practice.
	\end{remark}
	
	\begin{proof}
		Inductive application of \Cref{prop:error_bound_1} yields
		\begin{equation}\label{eq:errbound_with_epsk}
			\norm{\initsol_n-u(t_n)} \leq \eta_0 +\kappa_{2\nqp}h^{2\nqp+1}n  + \frac{\rho}{1-\rho}\sum_{k=1}^n\varepsilon_k.
		\end{equation}
		Choosing $\varepsilon_k$ as prescribed and using the inequality $1+x\leq e^x$, which holds for any $x$, we obtain the error bound.
		
		For the rank bound, using \Cref{lem:local_inside} and \Cref{prop:error_bound_1} we get
		\begin{align}
			\begin{split}\label{eq:bound_vnun_by_en}
				\normJ{\fixedpn{n}-\exsoln{n}}&\leq\frac{\kappa_\nqp h^{\nqp+1}+\norm{\initsol_{n-1}-u(t_{n-1})}}{1-\rho} \\
				&\leq \frac1{1-\rho}\Bigl( \kappa_\nqp h^{\nqp+1}+\eta_0+\frac{\rho}{1-\rho}\sum_{k=1}^{n-1}\varepsilon_k+\kappa_{2\nqp}h^{2\nqp+1}(n-1) \Bigr) \\
				&= \frac{1-\rho}{4(1+\rho)}\frac{1-\nu}{3-2\nu}\varepsilon_n.
			\end{split}
		\end{align}
		Then applying \Cref{prop:local_error_rank} in the subinterval $\timeint_n$ we find that $\rank_\mat{(\compn{n}_{i,j})} \leq \overline{r}^{(n)}_{i}$, where
		\begin{equation*}
			\overline{r}^{(n)}_{i} := \frac{128}{(1-\rho)^2}\frac{\normJ{\exsoln{n}-\softt_{\alpha_{i,n}}\exsoln{n}}^2}{\alpha_{i,n}^2} + \frac{32/(1-\nu)^2}{\theta^2(1-\rho)^2}\frac{\normJ{\exsoln{n}-\softt_{\alpha_{i-1,n}}\exsoln{n}}^2}{\alpha_{i-1,n}^2} +\rank_\mat \Big(\softt_{\mat,\alpha_{i,n}/2}\exsoln{n} \Big)
		\end{equation*}
		We next recall that $\initsol_n=\cfpmap_n\compn{n}_{I_n,J_{I_n}}(nh)$. Then, for all $\mu=1,2,\dots,E$,
		\begin{align*}
			\rank_\mat(\initsol_n) &= \rank_\mat\Big(\cfpmap_n\compn{n}_{I_n,J_{I_n}}(nh)\Big) \leq \rank_\mat(\initsol_{n-1}) + \nqp\rank_\mat(V)\rank_\mat\Big(\compn{n}_{I_n,J_{I_n}}\Big) \\&\leq \rank_\mat(\initsol_{n-1}) + \nqp\rank_\mat(V)\overline{r}^{(n)}_{I_n},
		\end{align*}
		where we used \cite[Lemma 22]{BDS:25} in the first inequality. Iterating this process we obtain
		\begin{equation*}
			\rank_\mat(\initsol_n) \leq \rank_\mat(\initsol_{0}) + \nqp\rank_\mat(V)\sum_{k=1}^n\overline{r}^{(k)}_{I_k}
		\end{equation*}
		and inserting the definition of $\overline{r}^{(k)}_{I_k}$ yields the final result.
	\end{proof}
	
	\begin{remark}
		Combining \eqref{eq:bound_vnun_by_en} and \eqref{eq:bound_en_by_unvn} in the subinterval $\mathcal{I}_n$, we can derive an upper bound for the stopping tolerance $\varepsilon_n$ in terms of the soft-thresholding error of the exact solution: for $i=0,1,\dots,I_n$ and all $n$, it holds that
		\begin{equation*}
			\varepsilon_n \leq 4\frac{1+\rho}{1-\rho}\normJ{\exsoln{n}-\softt_{\alpha_{i,n}}\exsoln{n}} + \frac{2(1+\rho)}{(1-\nu)(1-\rho)}\normJ{\exsoln{n}-\softt_{\alpha_{i-1,n}}\exsoln{n}} \,.
		\end{equation*}
		Inserting this into \eqref{eq:errbound_with_epsk} and choosing $i=I_n$ we obtain the alternative error bound
		\begin{equation*}
			\norm{\initsol_n-u(t_n)}\leq\eta_0 + \kappa_{2\nqp}h^{2\nqp}t_n+\frac{2\rho(1+\rho)}{(1-\rho)^2}\sum_{k=1}^n\left(2\normJ{\exsoln{k}-\softt_{\albar_{k}}\exsoln{k}} + \frac{\normJ{\exsoln{k}-\softt_{\altil_{k}}\exsoln{k}}}{1-\nu}\right).
		\end{equation*}
		
	\end{remark}
	
	\subsection{Quasi-optimality of rank estimates}
	We now show that the rank bounds obtained in the previous Sections are quasi-optimal in the sense defined in Section~\ref{sec:accuracy_bounds}. Our reference is the rank of the exact solution $\exsol=(u_q)_{q=1,\dots,\nqp}$ defined, for $\mu=1,2,\dots,E$ and $\alpha>0$, by
	\begin{equation*}
		r_{\mu,\alpha}(\exsol) := \max_{q = 1,\ldots, Q} \# \{k \in\mathbb{N} : \sigma_{\mu,k}(u_q)>\alpha \}.
	\end{equation*}	
	It is well known (see, for example, \cite[Prop.~3.6]{BS17} and the references given there) that, in the case of algebraic and exponential decay of the singular values of $\exsol$, with the notation introduced in \Cref{subsec:sv_decay}, we have
	\begin{equation}\label{eq:ranks_decay}
		r_{\mu,\alpha}(\exsol)\lesssim M(\exsol)^p\alpha^{-p} \qquad \text{ and } \qquad r_{\mu,\alpha}(\exsol)\lesssim(1+\lvert\ln\alpha\rvert)^{\frac{1}{\beta}},
	\end{equation}
	respectively, where $M(\exsol):=\max_\mu\max_q\lvert\sigma_\mu(u_q)\rvert_{\ell^{p,\infty}}$ and the constants depend only on $p$ and only on $C$, $c$ and $\beta$, respectively. Therefore, our goal is to show that whenever the growth of best approximation ranks is bounded in such a form, analogous estimates hold for the ranks of the approximations produced by our scheme, with different constants depending on the same quantities and, possibly, on low-order polynomial functions of the dimension.
	
	\begin{prop}
	\label{prp:quasiopt}
	For $n=1,2,\dots,N$ and $\mu = 1,2,\ldots, E$, we have $\max_{i,j} \rank_\mu(\mathbf v_{i,j}^{(n)}) \lesssim R(E, n)$, where 	
	\begin{equation*}
		{R}(E,n)=\begin{cases}
			E^{2+\frac{2}{s}} M(\exsoln{n})^{\frac{1}{s}}\varepsilon_n^{-\frac{1}{s}} & \text{if $\sigma_\mat(\exsoln{n})\in\ell^{p,\infty}$, $s = \frac1p - \frac12$,} \\
			E^2(1+\lvert\ln E\rvert)^{\frac{1}{\beta}} (1+\lvert\ln\varepsilon_n\rvert)^{\frac{1}{\beta}} & \text{if $\sigma_{\mat,j}(\exsoln{n})\leq Ce^{-cj^\beta}$ for all $j$.}
		\end{cases}
	\end{equation*}
	\end{prop}
	
	\begin{proof}
	We suppress the dependence on $n$ to simplify notation. 
	By the local rank bound of \Cref{prop:local_error_rank},
	\begin{equation}\label{eq: rank sum}
		\rank_\mu(\mathbf v_{i,j}) \lesssim \frac{\| \mathbf{u} - \mathcal{S}_{\alpha_i} \mathbf{u} \|_Q^2}{\alpha_i^2} + \frac{\| \mathbf{u} - \mathcal{S}_{\alpha_{i-1}} \mathbf{u} \|_Q^2}{\alpha_{i-1}^2} + \rank_\mu(\mathcal S_{\mu, \alpha_i/2} \mathbf u),
	\end{equation}
	where the omitted constant depends on $\rho$, $\nu$ and $\theta$.
	Moreover, as long as the stopping criterion is not satisfied, \eqref{eq:bound_en_by_unvn} gives
	\begin{align} 
		\begin{split}\label{eq: eps bound}
			\varepsilon &\leq 4 \frac{1 + \rho}{1 - \rho} \| \mathbf{u} - \mathcal{ S}_{\alpha_i} \mathbf{ u} \|_Q + 2 \frac{1+\rho}{(1-\nu)(1-\rho)}\| \mathbf{u} - \mathcal{ S}_{\alpha_{i-1}} \mathbf{ u} \|_Q \\
			&\lesssim \| \mathbf{u} - \mathcal{ S}_{\alpha_i} \mathbf{ u} \|_Q + \| \mathbf{u} - \mathcal{ S}_{\alpha_{i-1}} \mathbf{ u} \|_Q \\
			&\lesssim E  \min \Big\{ \| \mathbf{u} - \mathcal{S}_{\alpha_i} \mathbf{u} \|_Q, \| \mathbf{u} - \mathcal{S}_{\alpha_{i-1}} \mathbf{u} \|_Q \Big\},
		\end{split}
	\end{align}
	with a constant also depending on $\rho$, $\nu$ and $\theta$, and where we used \eqref{eq:stabilityE} to obtain the last inequality.

	Under the assumption that the singular values have algebraic decay, by \cite[Proposition 3.6]{BS17} we have
	$\| \mathbf u - \mathcal S_{\alpha}(\mathbf u) \|_Q \lesssim E M(\mathbf u)^{\frac{p}{2}} \alpha^{1-\frac{p}{2}}$,
	which for $s = \frac{1}{p}-\frac{1}{2}$ gives 
	\begin{equation*}
		\frac{\| \mathbf{u} - \mathcal{S}_{\alpha_i} \mathbf{u} \|_Q^2}{\alpha_i^2} \lesssim E^{\frac{4}{2-p}}  M(\mathbf u)^{\frac{2p}{2-p}} \| \mathbf{u} - \mathcal{S}_{\alpha_i} \mathbf{u} \|_Q^{-\frac{2p}{2-p}} = E^{\frac{4}{2-p}} \, M(\mathbf u)^{\frac{1}{s}} \| \mathbf{u} - \mathcal{S}_{\alpha_i} \mathbf{u} \|_Q^{-\frac{1}{s}}.
	\end{equation*}
	Analogously, we obtain
	\begin{equation*}
		\frac{\| \mathbf{u} - \mathcal{S}_{\alpha_{i-1}} \mathbf{u} \|_Q^2}{\alpha_{i-1}^2} \lesssim E^{\frac{4}{2-p}} \, M(\mathbf u)^{\frac{1}{s}} \| \mathbf{u} - \mathcal{S}_{\alpha_{i-1}} \mathbf{u} \|_Q^{-\frac{1}{s}}.
	\end{equation*}
	Using \eqref{eq:ranks_decay}, it further holds that 
	\begin{align*}
		r_{\mu,\alpha_i/2}(\exsol)=\rank_\mu(\mathcal S_{\mu, \alpha_i/2} \mathbf u) &\lesssim M(\mathbf u)^p \alpha_i^{-p} \lesssim M(\mathbf u)^p E^{\frac{1}{s}}  M(\mathbf u)^{\frac{p^2}{2-p}}\| \mathbf{u} - \mathcal{S}_{\alpha_i} \mathbf{u} \|_Q^{-\frac{1}{s}} \\
		&= E^{\frac{1}{s}} M(\mathbf u)^{\frac{1}{s}} \| \mathbf{u} - \mathcal{S}_{\alpha_i} \mathbf{u} \|_Q^{-\frac{1}{s}}.
	\end{align*}
	Inserting the last three inequalities into \eqref{eq: rank sum}, and using that $\frac{4}{2-p}=2+\frac{1}{s} > \frac{1}{s}$, we arrive at
	\begin{equation*}
		\rank_\mu(\mathbf v_{i,j}) \lesssim E^{2+ \frac{1}{s}}  M(\mathbf u)^{\frac{1}{s}} \min \Big\{ \| \mathbf{u} - \mathcal{S}_{\alpha_i} \mathbf{u} \|_Q, \| \mathbf{u} - \mathcal{S}_{\alpha_{i-1}} \mathbf{u} \|_Q \Big\}^{-\frac{1}{s}}.
	\end{equation*}
	Using \eqref{eq: eps bound}, this results in 
	\begin{equation*}%
		\rank_\mu(\mathbf v_{i,j}) \lesssim E^{2 + \frac{2}{s}} M(\mathbf u)^{\frac{1}{s}} \varepsilon^{-\frac{1}{s}},
	\end{equation*}
	which shows the stated rank bound in the case of algebraically decaying singular values.
	
	We apply the same strategy to the case of exponential decay of the singular values of $\exsol$.
	Under this assumption, by \cite[Proposition 3.6]{BS17},
	\begin{equation*}
		\| \mathbf u - \mathcal S_\alpha\mathbf u \|_Q \lesssim E (1+ \abs{\ln \alpha})^{\frac{1}{2 \beta}} \alpha \qquad \text{ and } \qquad r_{\mu,\alpha}(\exsol) \lesssim (1+\abs{\ln\alpha})^{\frac{1}{\beta}}.
	\end{equation*}
	We assume in the following that $\alpha_i,\alpha_{i-1}\in(0,1)$. We remark that this is actually not a restriction in the present work: in the case of the Schr\"odinger equation, the exact solution satisfies $\normJ{\exsol}=1$, which in turn implies that $\mathcal{S}_\alpha\exsol=0$ whenever $\alpha\geq 1$. Since $\alpha_i = \theta \alpha_{i-1}$, \eqref{eq: rank sum} then yields
	\begin{equation}\label{eq: interm rank exp}
		\rank_\mu(\mathbf v_{i,j}) \lesssim E^2 (1+ \abs{\ln \alpha_i})^{\frac{1}{ \beta}}.
	\end{equation}
	Moreover, \eqref{eq: eps bound} gives $\varepsilon \lesssim E^2 (1 + \abs{\ln \alpha_i})^{\frac{1}{2 \beta}}\alpha_i$,
	which, for $0 < \alpha_i < 1$, is equivalent to
	\begin{equation*}
		\abs{\ln \alpha_i} \lesssim \abs{\ln \varepsilon} + 2\ln E + \frac{1}{2\beta} \ln(1 +  \abs{ \ln \alpha_i } ).
	\end{equation*}
	Using the inequality $\ln(1+x) \leq C_\beta + \beta x$, with $C_\beta = \beta-\ln\beta-1 \geq 0$
	for all $x \geq 0$, we further obtain $	\abs{\ln \alpha_i} \lesssim \abs{\ln \varepsilon} + 2\ln E + \frac{1}{2} \abs{\ln \alpha_i}$
	and thus $\abs{\ln \alpha_i} \lesssim \abs{\ln \varepsilon} + \ln E$,
	with a constant that depends on $\rho$, $\theta$, $C$, $c$ and $\beta$ but is independent of $\alpha_i$, $\varepsilon$ and $E$.
	Inserting this back into \eqref{eq: interm rank exp} yields
	\begin{align}
	\label{eq:optimality_local_exp}
			\rank_\mu(\mathbf v_{i,j}) \lesssim E^2 (1+ \abs{\ln \varepsilon} + \ln E)^{\frac{1}{ \beta}} 
			\lesssim E^2 (1 +  \ln E)^{\frac{1}{\beta}}(1 + \abs{\ln \varepsilon})^\frac{1}{\beta}.
	\end{align}
	This is the stated bound in the case of exponentially decaying singular values.
	\end{proof}

	Compared to \eqref{eq:ranks_decay}, \Cref{prp:quasiopt} shows (local) quasi-optimality in the case of algebraic or exponential-type decay of singular values, in both cases with constants depending only polynomially on $d$ via $E$.
	These local optimality bounds can be easily extended to global bounds on the ranks of the approximations $\initsol_n$ computed at the times $t_n$. 
	
	\begin{cor} \label{cor:quasiopt}
	With the notation and assumptions of \Cref{prp:quasiopt}, for $\mu = 1,2,\ldots, E$ and all $n$, we have
	\begin{equation*}
		\rank_\mu(\initsol_n) \lesssim  \rank_\mat(\initsol_0) +  \frac{T}{h} \, Q \rank_\mu (V) \max_{k = 1, \dots, n} {R}(E,k).
	\end{equation*}
	\end{cor}
	
	\begin{proof}
	This follows directly from the bound on the rank of $\initsol_n$ given by \Cref{prop:global_error_rank} combined with \Cref{prp:quasiopt}.
	\end{proof}
	
	This means in particular that the obtained approximation ranks remain quasi‑optimal with respect to the best achievable ranks up to polynomial factors in the dimension and a linear dependence on the number of time steps.
	Note that the counterexamples given in \cite{DSZ26} indicate that when comparing ranks to those of best approximations of the exact solution $u$, a prefactor scaling as $h^{-1}$ can in general not be avoided without further assumptions on the problem -- transferred to our setting, this would mean that the factor $h^{-1}$ in the bound of Corollary \ref{cor:quasiopt} cannot be removed even when using positive recompression tolerances $\delta_n$.
	At the same time, here we do not require $h\to 0$ to enforce convergence, but rather aim at refinement of discretizations on each subinterval.

	\section{Numerical experiments}\label{sec:numexp}
    In this section, we assess the performance of the iterative thresholding integrator combined with a hierarchical tensor representation for high-dimensional Schr\"odinger initial value problems of the form \eqref{eq:schroedinger} using a test problem with verifiable reference solutions. We track the evolution of the hierarchical ranks and report the norm deviation $|\norm{u(t)}-1|$ and the relative energy error. While the test case is beyond what is covered by our analysis,  the results show that the method still performs well.
    
    In an example in four dimensions, the low-rank solution error is measured in the Frobenius norm against a reference solution represented as the full coefficient tensor. Constructing a comparable tensor in a second $64$-dimensional test is computationally prohibitive, but in this case the analytic Gaussian reference solution can still be evaluated pointwise. We therefore estimate the continuous $L^2$ error at the time interval boundaries by the Monte Carlo sampling described in \Cref{app:gaussian_reference}. All experiments were implemented in Julia~1.11.5 and performed on 16 CPU cores on a Dell PowerEdge M630 node equipped with two Intel Xeon E5-2660 v3 processors and approximately $250$~GB of memory.

    Following \cite{TC2015,RO2016}, on $\R^d$ we consider the bilinearly coupled oscillator Hamiltonian
    \begin{align}
    \label{eq:BCO_Hamiltonian}
     H = \sum_{i=1}^d \frac{\omega_i}{2}\big(-\partial_{x_i}^2+x_i^2 \big)
         + \sum_{j=1}^{d-1}\sum_{i>j} \alpha_{ij}x_ix_j,
    \end{align}
    where $\omega_i>0$ denotes the frequency of the $i$th harmonic oscillator, and $\alpha_{ij}>0$ denotes the coupling strength between the $i$th and $j$th degrees of freedom. Each coordinate is governed by a harmonic oscillator, while the quadratic cross terms describe bilinear interactions between distinct coordinates. Hamiltonians of this type arise, for example, in molecular vibration analysis and coupled lattice dynamics. This class of problems goes beyond our theory in that it involves unbounded potentials on $\R^d$.  In the experiments below, we set $\omega_j=\sqrt{\frac{j}{2}}$ and $\alpha_{ij}=\frac{1}{10}$. 
    
    For the spatial discretization, we employ a Hermite basis and retain $K$ basis functions in each coordinate direction. The spatial semidiscretization of the $d$-dimensional Schr\"odinger equation with the Hamiltonian \eqref{eq:BCO_Hamiltonian} yields the finite-dimensional system
    \begin{equation}
    \label{eq:semi_disc}
     i\partial_t\mathbf u(t)=(\mathbf H_1 + \mathbf H_2) \mathbf u(t)
    \end{equation}
    with the matrices
        \[
     \mathbf H_1=\sum_{i=1}^d\mathbf S_i,
     \qquad
     \mathbf H_2=\frac1{20}\left(\left(\sum_{i=1}^d\mathbf Q_i\right)^2-
     \sum_{i=1}^d\mathbf Q_i^2\right)
    \]
    where $\mathbf S_i=\frac{\omega_i}{2}I^{\otimes(i-1)}\otimes\mathbf S\otimes I^{\otimes(d-i)}$ and
    $ \mathbf Q_i=I^{\otimes(i-1)}\otimes\mathbf Q\otimes I^{\otimes(d-i)}$.
    Here, $\mathbf Q\in\R^{K\times K}$ is the symmetric tridiagonal matrix defined by
    $(\mathbf Q)_{j,j+1}=(\mathbf Q)_{j+1,j}
     =\sqrt{j/2}$ for $j=1,\ldots,K-1$,
    with all remaining entries equal to zero, and $\mathbf S=\operatorname{diag}(1,3,\ldots,2K-1)$. Note that a direct representation of the bilinear coupling involves $d(d-1)/2$ pairwise interaction terms, whereas \eqref{eq:semi_disc} expresses the coupling through operators of hierarchical rank bounded independently of $d$; in particular, $\sum_{i=1}^d \mathbf S_i$ and $\sum_{i=1}^d \mathbf Q_i$ are of rank two. We evaluate the interaction term by applying $\sum_i\mathbf Q_i$ twice and $\sum_i\mathbf Q_i^2$ once.
    
    The twisted variable used in these experiments differs slightly from that introduced in \eqref{eq:twisted_f}. 
    Instead of twisting only with the Laplacian, we use the full separable part $\mathbf H_1$.
    The transformed variable then satisfies
    \[
     \partial_t\left(\mathrm e^{it\mathbf H_1}\mathbf u(t)\right)=-i \mathrm e^{it\mathbf H_1}\mathbf H_2\mathrm e^{-it\mathbf H_1}\left(\mathrm e^{it\mathbf H_1}\mathbf u(t)\right).
    \]
    Because the exponential of $\mathbf H_1$ is a tensor product of one-dimensional exponentials, its application leaves the hierarchical ranks unchanged. We apply Algorithm~\ref{alg:basic} to approximate $\mathrm e^{it\mathbf H_1}\mathbf u(t)$.
    
   In addition to the recompression operations $\recomp_{\delta_n}$ at interval endpoints, we introduce recompression operations with prescribed relative error to improve the practical efficiency of the evaluation of the fixed-point mappings. Here we use a simplified version of the strategy used in \cite[Sec.~5.1]{BS17}. With
    $r_{i,j}^{(n)}:=\normJ{\fpmap_n\compn{n}_{i,j}-\compn{n}_{i,j}}$,
    for the next evaluation of $\fpmap_n$, we choose a recompression budget of $10^{-2}r_{i,j}^{(n)}$. Half of this budget is allocated to the application of the $Q$-operator terms and half to their accumulation. The associated tolerances are divided by $Q$ to distribute the budget over the $Q$ stages. 
    
    In all tests, we use $Q=10$, step size $h = \frac1{10}$ and $\theta=\frac12$.
    The inner iterations converge rather quickly in our tests. During the first few outer iterations, the stopping criterion is typically met after one inner step, whereas later outer iterations usually require two steps, so that $J_i\in\{1,2\}$. This favorable convergence is observed even for values of $h$ that do not ensure contractivity of the fixed point mapping on the full space, an effect that may be related to a beneficial interaction with the thresholding operations that may merit further investigation. However, in this regime, the iteration parameters need to be chosen heuristically; in particular, the factor in the threshold decrease criterion is chosen as $\frac35$. 
    
    \subsection{4D example}
    Let $\varphi_0(x)=\pi^{-1/4}e^{-x^2/2}$ and $\varphi_1(x)=\sqrt2x\varphi_0(x)$. We consider the initial datum
    \begin{equation}\label{eq:init_4d}
     u_0(x)=\sqrt{\frac{2}{5}}\left[\prod_{\mu=1}^4\varphi_0(x_\mu)
     +\frac12\sum_{1\leq i<j\leq4}\varphi_1(x_i)\varphi_1(x_j)
     \prod_{\mu\notin\{i,j\}}\varphi_0(x_\mu)\right].
    \end{equation}
    Equivalently, $u_0(x)=\sqrt{2/5}\,\pi^{-1}e^{-\frac12\norm{x}^2}(1+\sum_{i<j}x_ix_j)$. This can be represented exactly in HT format, where all leaf ranks are equal to $2$ and where the rank of the remaining matricization corresponding to $s = \{3,4\}$ is $4$. For this test, we use $T = 2$ and $K = 50$ and set $\varepsilon_n=10^{-4}$, $\delta_n=7.07\times10^{-5}$. Figure \ref{fig:results_4d_rank} displays how the hierarchical ranks evolve during the time integration, while Figure \ref{fig:results_4d_errors} shows the solution error in the Frobenius norm, the norm deviation, and the relative energy error at all stored snapshots.
    
    \begin{figure}[tbp]
    \centering
    \includegraphics[width=0.8\linewidth]{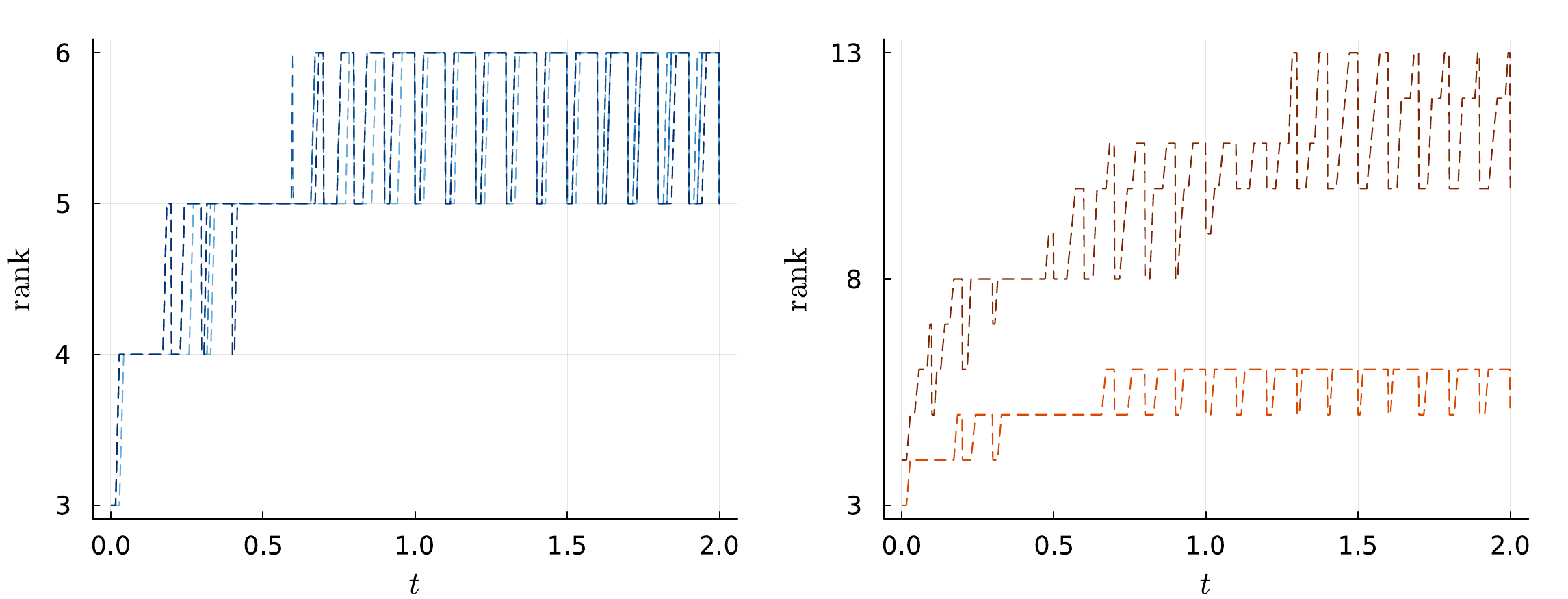}\\[-1mm]{\scriptsize (a) Legendre}\par\vspace{4mm}
    \includegraphics[width=0.8\linewidth]{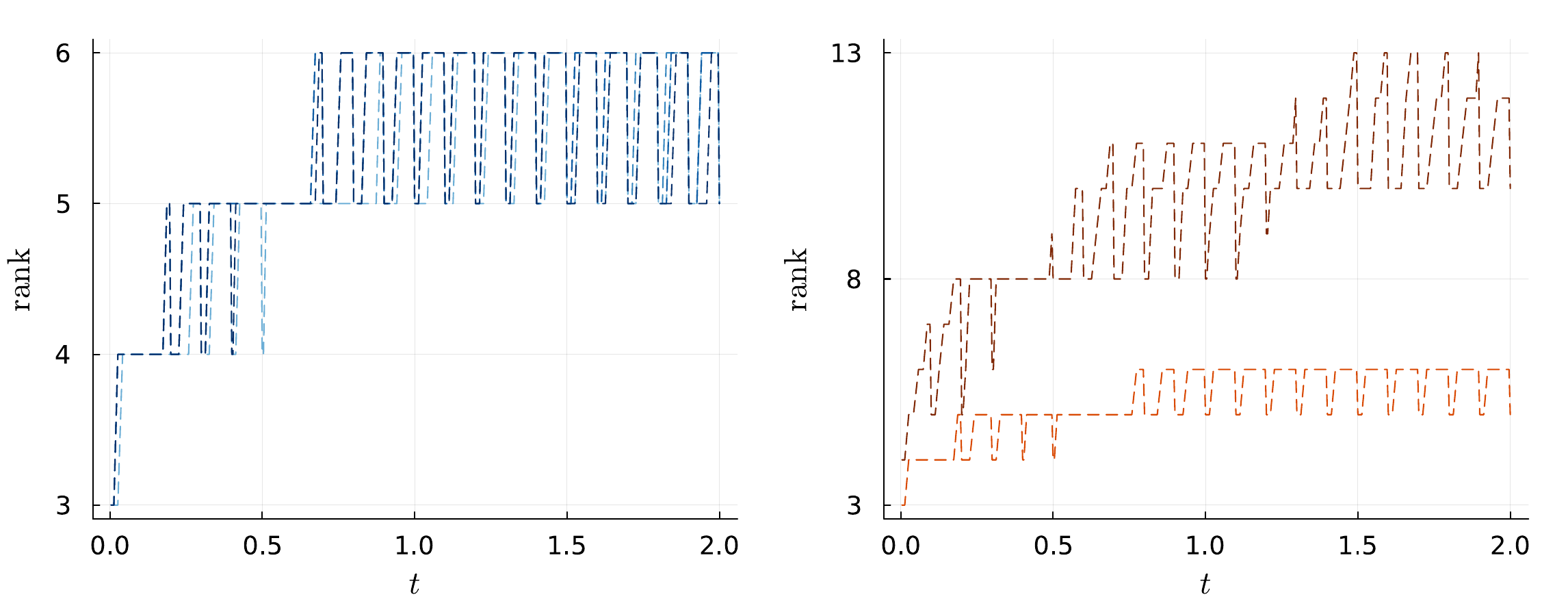}\\[-1mm]{\scriptsize (b) Lobatto}
    \caption{\small Rank evolution in 4D case for $u_0$ as in \eqref{eq:init_4d}. \emph{Blue lines, light to dark:} ranks of matricizations corresponding to $\{1\},\{2\},\{3\},\{4\}$; \emph{orange lines, light to dark:} ranks of matricizations corresponding to $\{2,3,4\}, \{3,4\}$. (a) Gauss--Legendre nodes; (b) Gauss--Lobatto nodes.}
    \label{fig:results_4d_rank}
    \end{figure}
    
    \begin{figure}[tbp]
    \centering
    \begin{minipage}{0.425\linewidth}\centering
    \includegraphics[width=.9\linewidth]{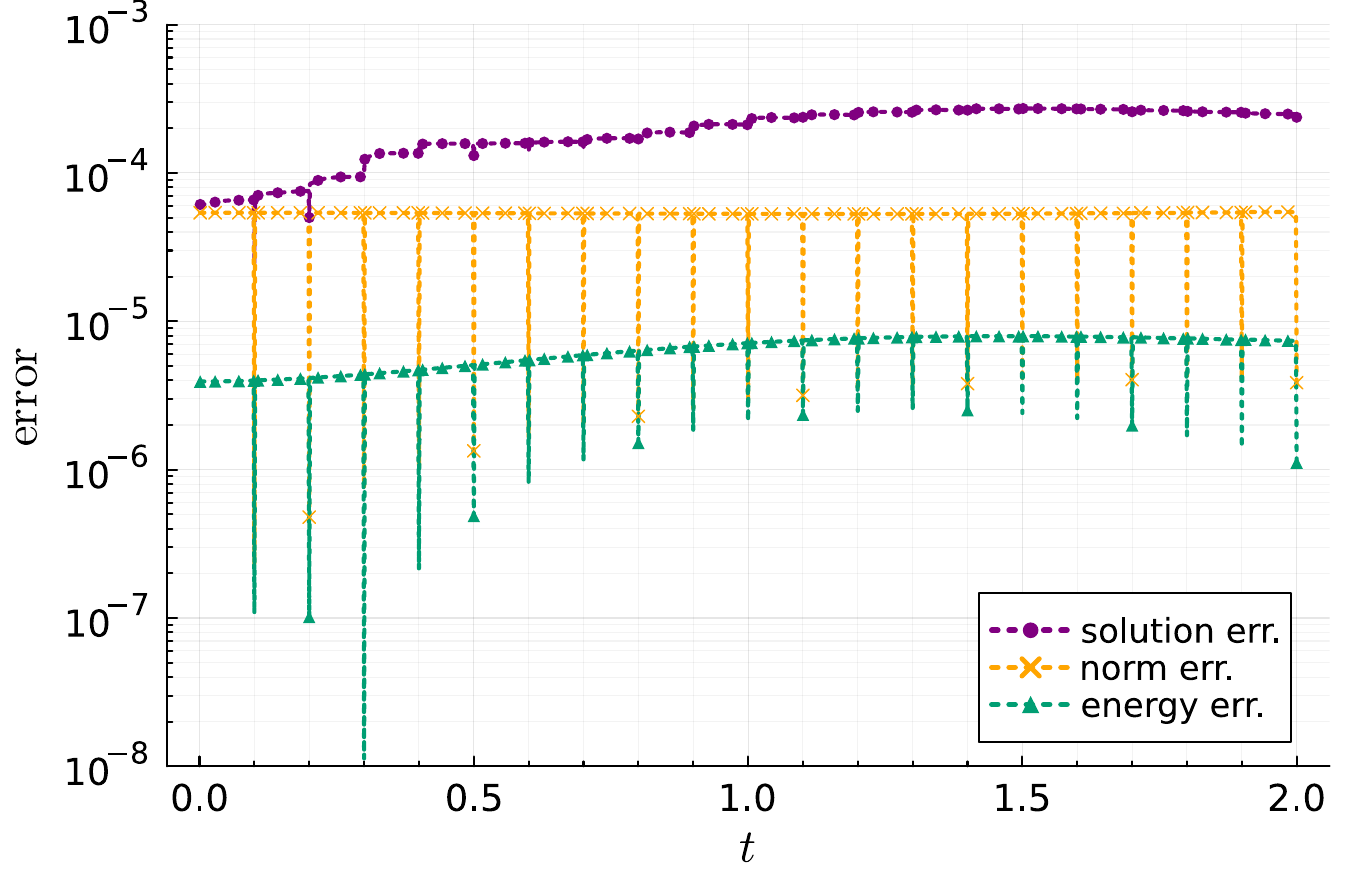}\\[-2mm]{\scriptsize (a) Legendre}
    \end{minipage}\hspace{2mm}
    \begin{minipage}{0.425\linewidth}\centering
    \includegraphics[width=.9\linewidth]{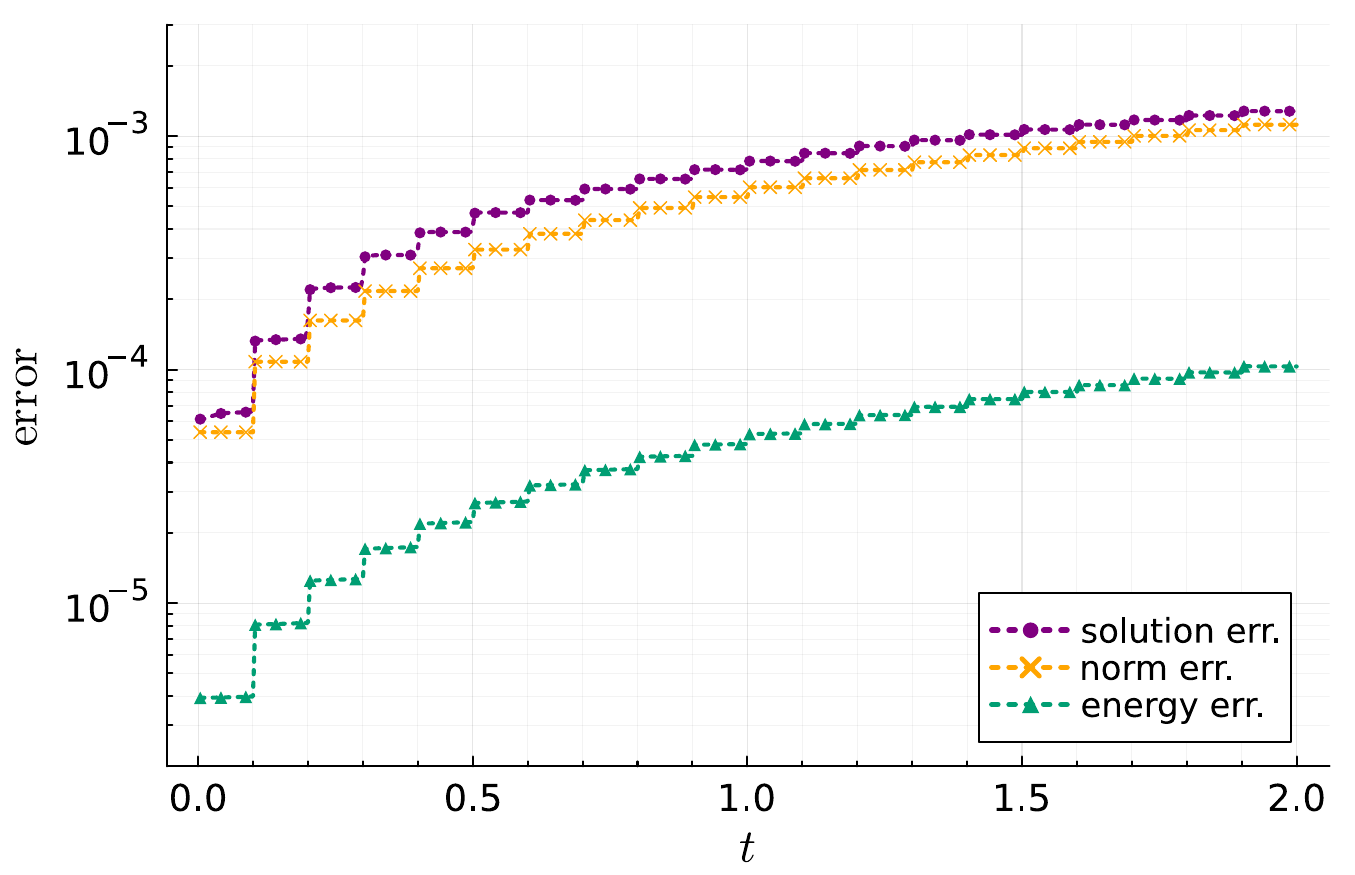}\\[-2mm]{\scriptsize (b) Lobatto}
    \end{minipage}
    \caption{\small Error evolution in 4D case for $u_0$ as in \eqref{eq:init_4d}. \emph{Purple line, circle markers:} Frobenius error of numerical solution; \emph{orange line, cross markers:} norm deviation; \emph{green line, triangle markers:} relative energy error. (a) Gauss--Legendre nodes; (b) Gauss--Lobatto nodes.}
    \label{fig:results_4d_errors}
    \end{figure}
    The reference solution is computed by applying $\exp(-it (\mathbf H_1 + \mathbf H_2))$ to the full coefficient tensor using a Krylov method with tolerance $10^{-10}$, followed by recompression with the same tolerance. Since the largest possible internal rank in the linear tree is $K^{\lfloor d/2\rfloor}$, for $K=50$, an internal rank can be as large as $K^2=2500$. In contrast, the stored histories for both collocation rules have a maximum leaf rank of $6$ and a maximum internal rank of $13$. The Gauss--Legendre test yields a maximum norm deviation of $5.45\times10^{-5}$, a maximum relative energy error of $7.95\times10^{-6}$, and a maximum solution error of $2.72\times10^{-4}$. The corresponding values for the Gauss--Lobatto test are $1.12\times10^{-3}$, $1.03\times10^{-4}$, and $1.28\times10^{-3}$.
    
   The Legendre nodes do not include $t_n$, and the higher-order accurate endpoint value is obtained through the additional evaluation $\cfpmap_n\compn{n}_{I_n,J_{I_n}}(t_n)$. For an exact fixed point, this map is isometric by \Cref{lem:iso_pres}. The remaining norm deviation arises from soft thresholding, residual-based recompressions, endpoint recompression, and the nonzero stopping residual. The simultaneous reduction in the energy error reflects the higher endpoint accuracy, although the exact energy conservation is not ensured in the twisted formulation. For the Lobatto method, $t_n$ is already the final stage. No additional application of $\Phi_n$ is required, and the corresponding endpoint oscillation is absent.
    
    \subsection{64D example with rank-one initial data}
    The normalized initial datum is
    \begin{equation}\label{eq:init_64d}
     u_0(x)=\prod_{\mu=1}^{64}\varphi_0(x_\mu)=\pi^{-16}e^{-\frac12\norm{x}^2},
    \end{equation}
with coefficients of hierarchical ranks one. Here, we choose $T = 1$ and $K = 32$ with $\varepsilon_n=5\times10^{-4}$, $\delta_n=1.77\times10^{-3}$. The results are shown in Figure \ref{fig:results_64d}.
    
    \begin{figure}[tbp]
    \centering
    \includegraphics[width=0.8\linewidth]{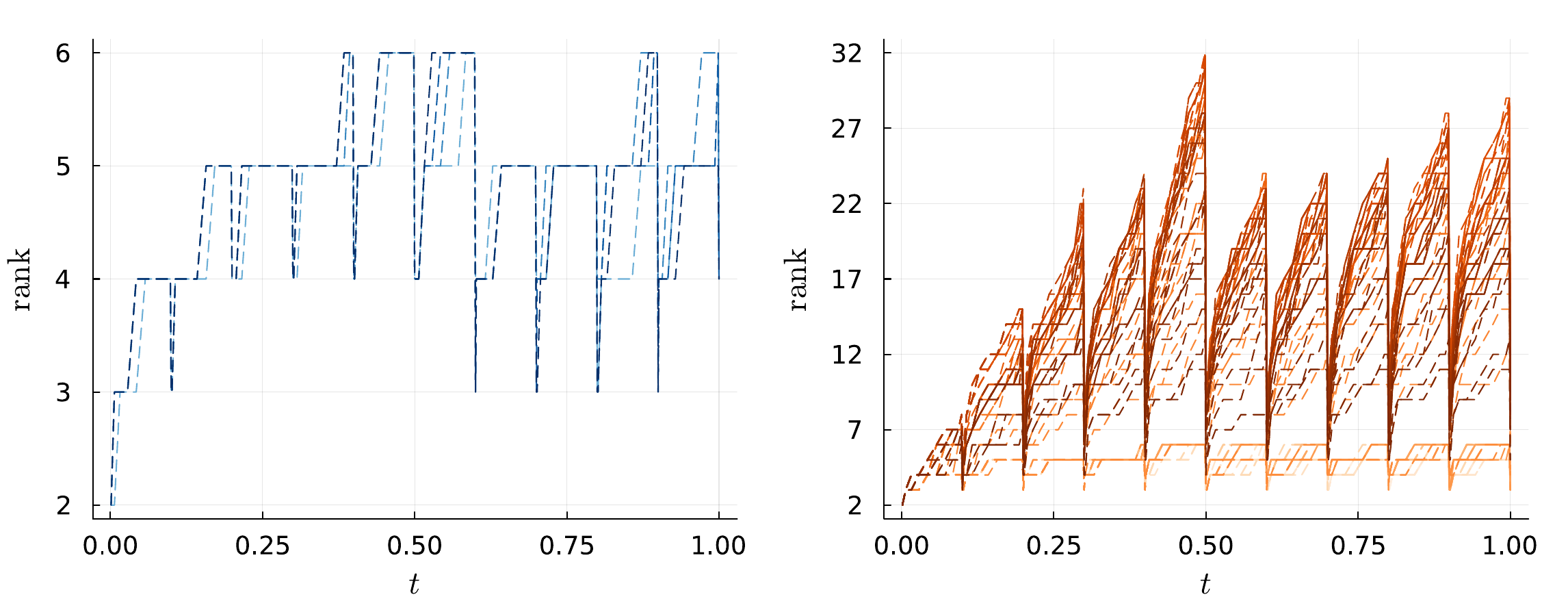}\\[-1mm]{\scriptsize (a) Rank evolution}\par\vspace{4mm}
    \includegraphics[width=0.425\linewidth]{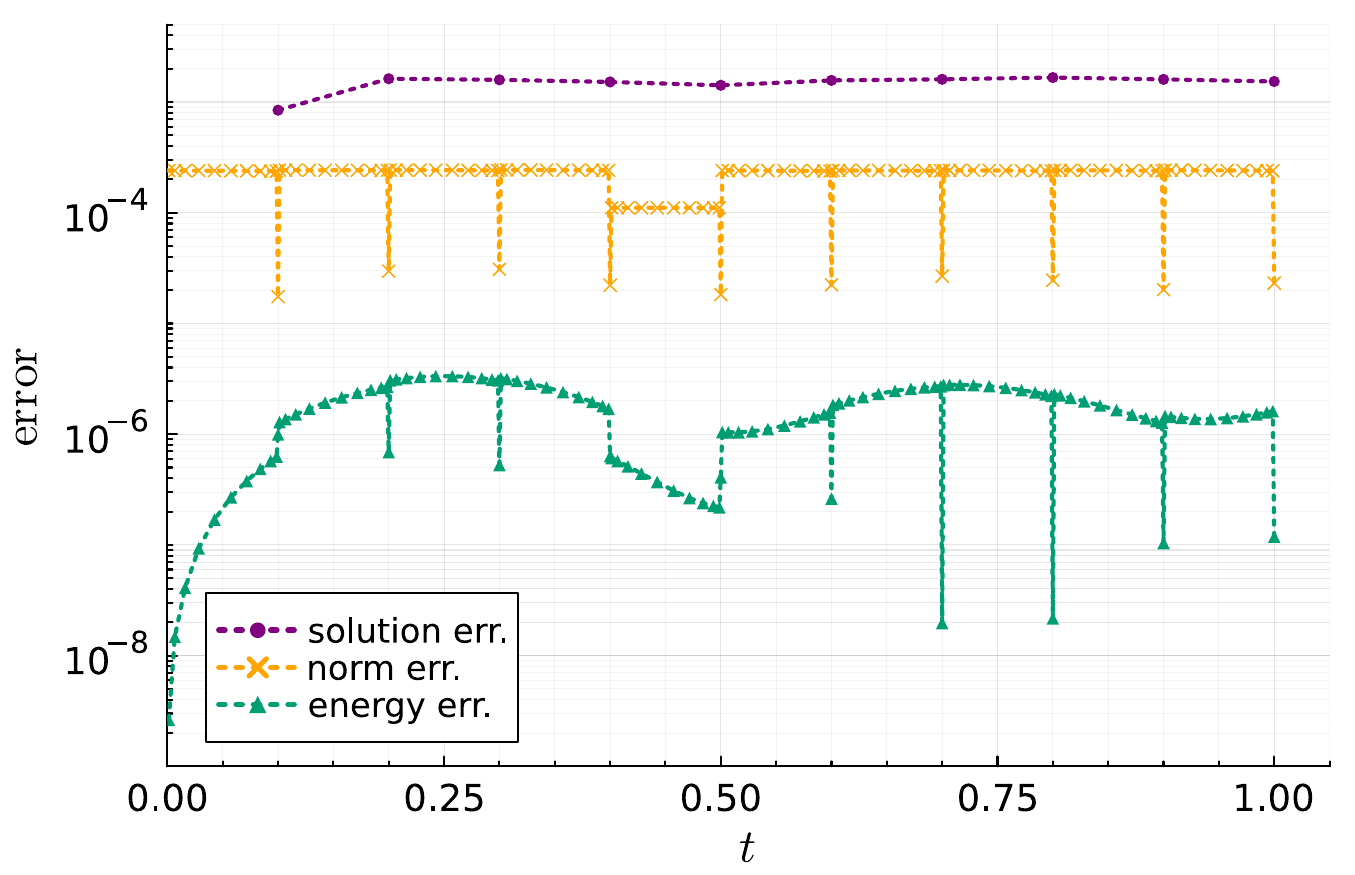}\\[-1mm]{\scriptsize (b) Error evolution}
    \caption{\small 64D results for $u_0$ as in \eqref{eq:init_64d}. (a) \emph{Blue lines, light to dark:}  ranks of matricizations corresponding to $\{1\},\{2\},\dots,\{64\}$; \emph{orange lines, light to dark:} ranks of matricizations corresponding to $\{2,\dots,64\}, \{3,\dots, 64\}, \dots, \{63, 64\}$. (b) \emph{Purple line, circle markers:} Monte Carlo estimate of continuous $L^2$ error at all $t_n$; \emph{orange line, cross markers:} norm deviation and \emph{green line, triangle markers:} relative energy error, both at all stored snapshots.}
    \label{fig:results_64d}
    \end{figure}
    The continuous reference solution is constructed by applying a Bogoliubov normal-mode transformation to the quadratic Hamiltonian and propagating the resulting Gaussian analytically. At each time-step endpoint, the numerical solution is transformed back to the physical variable and evaluated pointwise in its HT Hermite representation. Its continuous $L^2$ error is then estimated using $10^5$ samples from the normalized density of the Gaussian reference (see \Cref{app:gaussian_reference}). For $K=32$, the internal ranks can be as large as $K^{32}=32^{32}$. By comparison, the stored history has a maximum leaf rank of $6$ and a maximum internal rank of $32$. The rank reductions at the interval endpoints result from recompression. Over all $110$ stored collocation and endpoint snapshots, the maximum norm deviation and relative energy error are $2.43\times10^{-4}$ and $3.34\times10^{-6}$, respectively. The maximum estimated endpoint error is $1.66\times10^{-3}$ at $t=0.8$, and the estimate at $t=1$ is $1.53\times10^{-3}$.
	
	\section{Conclusion}
    In this paper, we have presented a time integration scheme for the high-dimensional, linear, time-dependent Schrödinger equation based on hierarchical tensor approximation. Building on previous analysis of the matrix setting, we have developed a corresponding framework for hierarchical tensors that avoids the curse of dimensionality while retaining rigorous performance guarantees.
	
	The proposed method combines a fixed-point iteration with thresholding of intermediate quantities, where the threshold is decreased adaptively to balance the approximation error and the tensor ranks. In this sense, our strategy can be regarded as an intermediate approach between dynamical low-rank approximation and step-truncation techniques. We have established quasi-optimality of the resulting ranks, with constants exhibiting only low-degree polynomial dependence on the problem dimension. 
	Numerical experiments in dimensions up to 64 demonstrate the effectiveness of the proposed approach also in tests that go beyond the theory.
	
	While the time-dependent Schr\"odinger equation served as the motivating application, the underlying arguments are not specific to this problem. Therefore, this work provides several interesting directions for future research. 
	A natural extension of the strategy outlined in this paper consists in its application to other classes of problems, such as nonlinear and parabolic evolution equations.
	We also mention the possibility of combining the present approach with time-stepping strategies designed for the regime of vanishing time steps. Favorable rank bounds then depend on further assumptions on the properties of the problem, and the analysis of such schemes remains to be investigated in future work.
	
	\section*{Acknowledgements}
	
	The authors thank Dante Kennes for pointing out the explicit solution in Appendix \ref{app:gaussian_reference}.	
	ChatGPT 5.5 and 5.6 were used for literature research, proofreading and coding assistance.
	Funded by the European Union (ERC, COCOA, 101170147). Views and opinions expressed are however those of the authors only and do not necessarily reflect those of the European Union or the European Research Council. Neither the European Union nor the granting authority can be held responsible for them. Financial support from RWTH Aachen Profile Area Modeling and Simulation Sciences (MSS) for P.S.\ is gratefully acknowledged.
	
	\bibliographystyle{plain}
	\bibliography{BJSVtensorthreshint}
	
	\begin{appendix}
	\section{Properties of hierarchical tensor soft thresholding}\label{app:ststab}

	\begin{lemma} \label{L:stability}
		For $0 < \alpha \leq \beta$, the bound \eqref{eq:stabilityE} holds true. 
	\end{lemma}
	
	\begin{proof}
		By \cite[Lemma 10]{BDS:25}, 
				\begin{equation*}
			\normJ{\mathbf u - \cS_{\mu,\alpha}\mathbf u } \leq \normJ{\mathbf u - \cS_{\mu,\beta}\mathbf u} \leq \frac{\beta}{\alpha} \normJ{ \mathbf u - \cS_{\mu,\alpha}\mathbf u }
		\end{equation*}
		for all $\mu = 1,2,\dots, E = 2d-3$. We then apply \cite[Lemma 3.4]{BS17}: with $d_\mu^\alpha(\mathbf u) := \normJ{ \mathbf u - \cS_{\mu,\alpha}\mathbf u}$, 
		\begin{equation*}
			\normJ{ \mathbf u - \cS_\alpha\mathbf u } \leq E \max_{\mu=1,\dots,E} d_\mu^\alpha(\mathbf u) \leq E \max_{\mu=1,\dots,E} d_\mu^\beta(\mathbf u) \leq E \normJ{ \mathbf u - \cS_\beta\mathbf u}
		\end{equation*}
		and
		\begin{equation*}
			\normJ{ \mathbf u - \cS_\beta\mathbf u } \leq E \max_{\mu=1,\dots,E} d_\mu^\beta(\mathbf u) \leq E \frac{\beta}{\alpha} \max_{\mu=1,\dots,E} d_\mu^\alpha(\mathbf u) \leq E \frac{\beta}{\alpha} \normJ{ \mathbf u - \cS_\alpha\mathbf u },
		\end{equation*}
		where we used again \cite[Lemma 10]{BDS:25} to obtain the respective second inequalities.
	\end{proof}

	\section{Gaussian reference solution and Monte Carlo error estimator}\label{app:gaussian_reference}
    Since \eqref{eq:BCO_Hamiltonian} is quadratic, the solution associated with \eqref{eq:init_64d} remains Gaussian. With $p=-i\nabla$, the Hamiltonian can be written in the form
    \[
    H=\textstyle\frac12 \displaystyle p^\top\Omega p+ \textstyle\frac12 \displaystyle x^\top A x,\qquad
    \Omega=\operatorname{diag}(\omega_1,\ldots,\omega_d),\qquad
    A=\Omega+\textstyle\frac1{10} \displaystyle(\bbone\bbone^\top-I).
    \]
    Let $S=\Omega^{\frac12}$ and with the eigenvalue decomposition $SAS=O\operatorname{diag}(\nu_1^2,\ldots,\nu_d^2)O^\top$, let $C=SO$.
    The Bogoliubov canonical transformation $q=C^{-1}x$ and $p = C^{-\top}P$, reduces $H$ to independent oscillators with frequencies $\nu_j$. We refer to \cite{Kustura2019} for more general canonical transformations.
    
    For $\psi(0,x)=a_0e^{-\frac12x^\top G_0x}$, define $B_0=C^\top G_0C$ and
    \begin{align*}
    D_c(t)&=\operatorname{diag}(\cos(\nu_jt)),&
    D_s(t)&=\operatorname{diag}(\sin(\nu_jt)),\\
    Z(t)&=D_c(t)+i\operatorname{diag}(\nu_j^{-1})D_s(t)B_0,&
    B(t)&=\bigl(D_c(t)B_0+i\operatorname{diag}(\nu_j)D_s(t)\bigr)Z(t)^{-1}.
    \end{align*}
    Let $R=\operatorname{diag}(\nu_j)^{-\frac12}B_0\operatorname{diag}(\nu_j)^{-\frac12}$ and $\mathcal K=(I-R)(I+R)^{-1}$.
 Then writing $G(t)=C^{-\top}B(t)C^{-1}$,
    the Gaussian reference is
    \begin{equation}\label{eq:gaussian_reference}
    \psi(t,x)=a(t)\exp\left(-\frac12x^\top G(t)x\right),
    \end{equation}
    where the amplitude is computed directly as
    \[
    a(t)=a_0\exp\biggl(
    -\frac{it}{2}\sum_{j=1}^d\nu_j
    -\frac12\operatorname{tr}\operatorname{Log}(I+\operatorname{diag}(\mathrm e^{-2i\nu_jt})\mathcal K)
    +\frac12\operatorname{Log}\det(I+\mathcal K)\biggr),
    \]
    and $\operatorname{Log}$ denotes the principal matrix logarithm \cite{CR06}. For \eqref{eq:init_64d}, $G_0=I$ and $a_0=\pi^{-d/4}$. This formula gives $a(0)=a_0$. We next use this reference solution to estimate the continuous $L^2$-error
    \begin{equation*}
        e(t_n)^2=\int_{\R^d}\left|\psi(t_n,x)-u_{\rm num}(t_n,x)\right|^2\,dx
        =N_{\rm ref}(t_n)^2\int_{\R^d}\left|1-\frac{u_{\rm num}(t_n,x)}{\psi(t_n,x)}\right|^2\rho_{t_n}(x)\,dx,
    \end{equation*}
    where
    \[
    N_{\rm ref}(t)^2
    =|a(t)|^2\frac{\pi^{d/2}}{\sqrt{\det(\operatorname{Re}G(t))}},
    \qquad
    \rho_t(x)=\frac{|\psi(t,x)|^2}{N_{\rm ref}(t)^2}
    =\frac{\sqrt{\det(\operatorname{Re}G(t))}}{\pi^{d/2}}
    \mathrm e^{-x^\top\operatorname{Re}G(t)x}.
    \]
    This expectation can be estimated by Monte Carlo sampling from $\rho_{t_n}$. If $U(t)^\top U(t)=\operatorname{Re}G(t)$ is a Cholesky factorization and $z\sim\mathcal N(0,I)$, then
    $X= 2^{-1/2} U(t)^{-1}z$
    has density $\rho_t$. The twisted numerical solution is transformed back to the physical variable and then evaluated at $X$. For independent samples $X_1,\ldots,X_{N_{\rm MC}}$, we use
    \[
    \widehat e(t_n)=
    \left[
    \frac{N_{\rm ref}(t_n)^2}{N_{\rm MC}}
    \sum_{m=1}^{N_{\rm MC}}
    \left|1-\frac{u_{\rm num}(t_n,X_m)}{\psi(t_n,X_m)}\right|^2
    \right]^{\frac12}
    \]
    to approximate $e(t_n)$ with $N_{\rm MC}=10^5$. 
	
	\end{appendix}
	
\end{document}